\documentclass[12pt]{article}
\usepackage{amsmath,amssymb,amsthm,url}
\usepackage{color}

\newcommand{\jonly}[1]{}
\newcommand{\aronly}[1]{#1}

\input{xy}
\xyoption{all}
\def\Z{{\mathbb Z}} \def\R{{\mathbb R}}  \def\C{{\mathbb C}}

\def\coker{\mathop{\fam0 coker}}

\def\di{\mathop{\fam0 div}}

\def\ex{\mathop{\fam0 ex}}

\def\id{\mathop{\fam0 id}}
\def\inc{\mathop{\fam0 i}}
\def\Int{\mathop{\fam0 Int}}

\def\im{\mathop{\fam0 im}}
\def\lk{\mathop{\fam0 lk}}

\def\mo{\mathop{\fam0 mod}}
\def\pr{\mathop{\fam0 pr}}

\def\Tors{\mathop{\fam0 Tors}}

\def\ker{\mathop{\fam0 ker}}
\def\ord{\mathop{\fam0 ord}}

\def\t{\widetilde}
\def\mk{\# }
\newcommand{\capM}[1]{\cap_{#1}}
\newcommand{\de}{\partial}
\newcommand{\nus}{\de\nu}

\long\def\comment#1\endcomment{}

    \theoremstyle{theorem}
         \newtheorem{Theorem}{Theorem}[section]
         \newtheorem{Lemma}[Theorem]{Lemma}
         \newtheorem{Corollary}[Theorem]{Corollary}

\theoremstyle{definition}

\newtheorem{Conjecture}[Theorem]{Conjecture}
\newtheorem{Remark}[Theorem]{Remark}

\newtheorem{Example}[Theorem]{Example}
\newtheorem{Problem}[Theorem]{Problem}
\newtheorem{Addendum}[Theorem]{Addendum}

\begin{document}

%\newpage
\title{\Large Embeddings  of non-simply-connected 4-manifolds in 7-space. II.\\ On the smooth classification}

\author{D. Crowley\thanks{The University of Melbourne, Australia.
Email: \texttt{dcrowley@unimelb.edu.au}. Web: \texttt{www.dcrowley.net} .} and A. Skopenkov\thanks{Moscow Institute of Physics and Technology, and Independent University of Moscow, Russia. Email: \texttt{skopenko@mccme.ru}. Web: \texttt{https://users.mccme.ru/skopenko/ .}
Supported in part by the Russian Foundation for Basic Research Grants No. 15-01-06302 and 19-01-00169, by Simons-IUM Fellowship and by the D. Zimin Dynasty Foundation.}
}

\date{}

\maketitle

\begin{abstract}
We work in the smooth category.
Let $N$ be a closed connected orientable 4-manifold with torsion free $H_1$, where $H_q := H_q(N; \Z)$.
Our main result is {\it a readily calculable classification of embeddings $N\to\R^7$ up to isotopy},
with an indeterminancy.
Such a classification was only known before for $H_1=0$ by our earlier work from 2008.
Our classification is complete when $H_2=0$ or when the signature of $N$ is divisible neither by 64 nor by 9.

The group of knots $S^4\to\R^7$ acts on the set of embeddings $N\to\R^7$ up to isotopy by embedded connected sum.
In Part I we classified the quotient of this action.
The main novelty of this paper is the description of this action for $H_1\ne0$, with an indeterminancy.

Besides the invariants of Part I, detecting the action of knots involves a refinement of the Kreck invariant from our work of 2008.

%Besides the invariants of Part I, the classification involves a refinement of the Kreck invariant
%from our work of 2008 which detects the action of knots when $H_1 = 0$.

For $N=S^1\times S^3$ we give a geometrically defined 1--1 correspondence between the set of isotopy classes of embeddings and a certain explicitly defined quotient of the set $\Z\oplus\Z\oplus\Z_{12}$.
\end{abstract}

\tableofcontents

%\section{Introduction}\label{s:smmain}

\section{Overview and main results}\label{s:intro-int}

We consider {\it smooth} manifolds, embeddings and isotopies.
For an $n$-manifold $P$ denote by $E^m(P)$ the set of isotopy classes of embeddings $P\to S^m$.
The group $E^m(S^n)$ acts on $E^m(P)$ by embedded connected sum \cite[\S1.1]{CS16}, \cite[\S4]{Sk16c}.
Denote this action by $\#$ and its quotient by $E_{\mk}^m(P)$.

\begin{Remark}[The action of knots in general] \label{r:general}
If the quotient $E_{\mk}^m(P)$ is known for a closed $n$-manifold $P$, the description of $E^m(P)$ is reduced
to the {\it determination of the orbits} of the embedded connected sum action of $E^m(S^n)$ on $E^m(P)$.
For a general closed $n$-manifold $P$ describing the  action by a non-zero group of knots
$E^m(S^n)$ on $E^m(P)$ is a non-trivial task.
%This determination is just as difficult as description of $E^m(P)$:
For the cases when the quotient $E_{\mk}^7(N)$ coincides with the set of PL embeddings up to PL isotopy, the quotient has been known since 1960s \cite{Sk16e, Sk16f, Sk16t}.
However, until recently no description of the action (or, equivalently, no classification of $E^m(P)$) was known for $E^m(S^n)\ne0$ and $P$ not a disjoint union of homology spheres.
For recent results see \cite{Sk08', Sk10, CS11}.
On the other hand, the description of the action in \cite{CRS07, Sk11, CRS11, Sk15} is not hard,
the hard part of the cases considered there is rather the description of the quotient $E_{\mk}^m(P)$.

There are non-isotopic embeddings $g_1,g_2\colon S^2\to S^4$ and an embedding $f\colon\R P^2\to S^4$ such that $[f\#g_1]\ne[f\#g_2]$ \cite{Vi73}; i.e. the action of the monoid $E^4(S^2)$ on $E^4(\R P^2)$ is not free.

Various authors have studied the analogous connected sum action of the group of homotopy $n$-spheres on the set of smooth $n$-manifolds homeomorphic to given manifold; see for example \cite{Le70, Sc73, Wi74} and references there.
%cf.~Remark \ref{r:inman}.
\end{Remark}

More motivation and background for this paper may be found in \cite{Sk16c, Sk16f} and Part I \cite[\S1]{CS16}.
%We adopt the notation and setting of \cite[\S1.1]{CS16}.
%and assume the reader is familiar with the results of \cite[\S1.2]{CS16}.
%In particular recall that
In this paper $N$ is a closed connected orientable $4$-manifold and $H_q:=H_q(N;\Z)$.
We present a readily calculable classification (in the sense of \cite[Remark 1.2]{Sk16c}) of $E^7(N)$ when $H_1$ is torsion free (up to an indeterminancy in certain cases).
See Theorems \ref{t:s1s3sm}, \ref{t:gensm} and Corollaries \ref{t:actioni}, \ref{c:gen} below.
Our classification is complete when $H_2=0$ (see Theorem \ref{t:gensm} and Corollary \ref{c:gen}.b) or when the signature of $N$ is divisible neither by 64 nor by 9 (see Theorem \ref{t:gensm} and Corollary \ref{t:actioni}).
The classification requires finding a complete set of invariants and constructing embeddings realizing particular values of these invariants.
The invariants we use are described in \cite[Lemma 1.3, \S2.2, \S2.3]{CS16} and \S\ref{s:defandplan-eta}.
%\ref{l:calam}, \S\ref{s:defandplan-kl}, \S\ref{s:defandplan-b}
An overview of the proof of their completeness is given in \cite[\S1.4]{CS16} and in
\jonly{\cite[Remark 7.4]{CS16o}}
\aronly{Remark \ref{r:action} below}
(using definitions recalled at the beginning of \S\ref{s:defandplan-eta}).
%\S\ref{s:intro-strat}.

The action of $E^7(S^4)\cong\Z_{12}$ on $E^7(N)$ was investigated in \cite{Sk10} and determined when $H_1=0$ in \cite{CS11}, which also classified $E^7(N)$ in this case.
In \cite{CS16} we described the quotient $E_{\mk}^7(N)$ when $H_1=0$.
Thus the main novelty of this paper is the description of this action for $H_1\ne0$.
%Cf.~Remark \ref{r:general}.

Denote by $q_{\mk}:E^7(N)\to E^7_{\mk}(N)$ the quotient map.

%In \cite{BH70} Bo\'echat and Haefliger classified $E^7_{\mk}(N)$ when $H_1(N; \Z) = 0$.
%The classification of embeddings of $N$ into $S^m$ {\it modulo knots} (i.e.~the description of $E^m_\mk(N)$
%defined just below) was known in the same cases.
%For $m\ge8$ the quotient map $E^m(N)\to E^m_\mk(N)$ is a 1--1 correspondence because $E^m(S^4)=0$ \cite[\S2]{Sk08},
%for $H_1=0$ and $m=7$ this classification is classical \cite{BH70}.

%\subsection{Main results}\label{s:smcla}

%The relation between Theorem \ref{t:s1s3sm} and \cite[Theorem 1.1]{CS1} is explained by
Let us state our main result for $N=S^1\times S^3$. 
For this identify $E^7(S^4)$ and $\Z_{12}$ by the isomorphism $\eta$ of \cite{CS11} (recalled in a more general situation in \S\ref{s:defandplan-eta}) and consider the following diagram (where the left triangle is not commutative):
%of surjections:
$$
\xymatrix{
\Z_{12}\times\Z^2 \ar[r]^{\pr_2} \ar[d]_{\#\times\tau} & \Z^2 \ar[d]^{\tau_{\mk}:=q_{\mk}\tau} \ar[dl]^{\tau}\\
E^7(S^1\times S^3)  \ar[r]_{q_\#} & E_{\mk}^7(S^1\times S^3).}
$$
The map $\tau$ is defined in \cite[\S1.2]{CS16}.
%\S\ref{s:defandplan-not}).
We define the map $\#\times\tau$ by $(\#\times\tau)(a,l,b):=a\#\tau(l,b)$.

\begin{Theorem} \label{t:s1s3sm}
The map $\#\times\tau\colon \Z_{12}\times\Z^2\to E^7(S^1\times S^3)$ is a surjection such that

(a) for different pairs $l,b$ the sets $P_{l,b}:=(\#\times\tau)(\Z_{12}\times(l,b))$  either are disjoint or coincide;

$$
(b)\qquad P_{l,b}=P_{l',b'}\quad\Leftrightarrow\quad(\ l=l' \quad\text{and}\quad b\equiv b'\mo2l\ );
$$
$$
(c)\qquad |P_{l,b}|=\begin{cases}12 & l\ne0\\ 2\gcd(b,6) &l=0.\end{cases}
$$
\end{Theorem}

In Theorem \ref{t:s1s3sm} the surjectivity of $\tau$ and Parts (a) and (b) follow from \cite[Theorem 1.1]{CS16}.
The new part of Theorem \ref{t:s1s3sm} is (c); this part follows from
Corollary \ref{c:gen}.b below (because for $l\ne0$ the group $\coker\overline l$ is finite, so $\di b=0$).
\aronly{Cf.~Addendum \ref{c:s1s3}.}
\jonly{Cf.~\cite[Addendum 7.3]{CS16o}.}

%The commutativity is clear. The surjectivity of $\pr_2$ and $q_\#$ is clear.
%The vertical maps are surjective by the surjectivity of Theorems \ref{t:s1s3sm}.
%The following remark is harder to state and to prove than Theorem \ref{t:s1s3sm}.
%Alternatively, the three properties easily follows from Addendum \ref{c:s1s3}.

\begin{Example}\label{e:com} There is an embedding $f\colon S^1\times S^3\to S^7$ with $f(N)\subset S^6$ and a pair of non-isotopic embeddings $g_1,g_2\colon S^4\to S^7$  such that $f\#g_1$ and $f\#g_2$ are isotopic.
\end{Example}

This example follows because $|P_{0,1}|=2$ by Theorem \ref{t:s1s3sm} and there is a representative of $\tau(0,1)$ whose image is in $S^6\subset S^7$ \cite[Lemma 2.18]{CS16}.
%\ref{l:cota}
Example \ref{e:com} shows the necessity of the simple-connectivity assumption in the following result
(which is \cite[The Effectiveness Theorem 1.2]{Sk10}):
%\footnote{It follows that the Effectiveness Theorem 1.2 of the earlier versions of \cite{Sk10} was false.}

{\it If $f\colon N\to S^7$ is an embedding of a spin simply-connected closed 4-manifold $N$, \ $f(N)\subset S^6$ and embeddings $g_1,g_2\colon S^4\to S^7$ are not isotopic, then $f\#g_1$ and $f\#g_2$ are not isotopic.}

Before stating our main result for the general case in Theorem \ref{t:gensm} below we state the following corollaries of it.

\begin{Corollary}[of Theorem \ref{t:gensm}.c; proved in \S\ref{s:defandplan-s1s3d}]\label{t:fa}
Let $N$ be a closed connected orientable 4-manifold with torsion free $H_1$.
Then the following statements are equivalent:

(i) for every embedding $f:N\to S^7$ and non-isotopic embeddings $g_1,g_2\colon S^4\to S^7$ the embeddings $f\#g_1$ and $f\#g_2$ are not isotopic
%(in other words,
(i.e. the action $\#$ of $E^7(S^4)$ on $E^7(N)$ is free);

(ii) $N$ is an integral homology $4$-sphere.
\end{Corollary}

The Bo\'echat-Haefliger invariant $\varkappa_{\mk}$ is defined in \cite[\S2.2]{CS16}, and we denote  $\varkappa=\varkappa_{\mk}q_{\mk}$:
%\S\ref{s:defandplan-kl})
$$\xymatrix{E^7(N) \ar[r]^{q_{\mk}} &  E_{\mk}^7(N) \ar[r]^(0.15){\varkappa_{\mk}} &
H_2^{DIFF}:=\{u\in H_2\ |\ \rho_2u=w_2^*(N),\ u\capM{N} u=\sigma(N)\}\subset H_2 }$$

\begin{Corollary}[of Theorem \ref{t:gensm}.c]\label{t:actioni}
Let $N$ be a closed connected orientable 4-manifold with torsion free $H_1$ and $f:N\to S^7$ an embedding.

(a) If $\varkappa(f)$ is neither divisible by $4$ nor by $3$, then for every embedding $g\colon S^4\to S^7$ the embeddings $f\#g$ and $f$ are isotopic.

(b) If $\varkappa(f)$ is divisible by 4 but neither by 8 nor by 3, then there is a non-trivial embedding $g_1\colon S^4\to S^7$ such that for every  embedding $g\colon S^4\to S^7$ the embedding $f\#g$ is isotopic to either $f$ or $f\#g_1$.
\end{Corollary}

Corollary \ref{t:actioni} follows from Corollary \ref{c:gen}.bc (or from Theorem \ref{t:gensm}.c and Addendum \ref{l:addeta} because $4\Z_{\gcd(\varkappa(f),24)}=0$ under the assumptions of Corollary \ref{t:actioni}).
The assumption of Corollary \ref{t:actioni}.a is automatically satisfied
when the signature of $N$ is divisible neither by 16 nor by 9.
%Cf.~Remark \ref{r:at}.a.

%\begin{Example} Let $\phi \colon \sqcup_{i=1}^n D^2 \times S^2 \to S^4$ be a smoothing embedding.
%The map $\phi$ defines a framed link in $S^4$.  The outcome of surgery on $\phi$ is a closed smooth
%connected orientable $4$-manifold $N$ with $H_1 \cong \Z^n$ and $H_2 = 0$. \end{Example}

Denote
$$
\widehat n:=\gcd(n,24).
$$
%\begin{Theorem} \label{t:1con} Let $N$ be a closed connected 4-manifold such that
{\it If $H_1=0$, then $\varkappa_{\mk}$ is 1--1 and $\varkappa=\varkappa_{\mk}q_{\mk}$ is surjective.
For each $u\in H_2^{DIFF}$ we have $|\varkappa^{-1}(u)|=\widehat u/\gcd(u,2)$.}
Here the first sentence is easily deduced from \cite{BH70}, see \cite[Remark 2.20.e]{CS16}.
The second sentence is proved in \cite{CS11}.
%\end{Theorem}

Our second main result is a generalization of the above statement to non-simply-connected 4-manifolds.
We use the convention on coefficients and the notation for characteristic classes and intersections in manifolds from \cite[\S1.2]{CS16}.
%We recall the following definitions of \cite[\S1.2]{CS16}.

\smallskip
{\bf Definition of $\di$, $B(H_3)$, $\overline l$, a symmetric pair, $K_{u,l}$, $C_{u,l}$ and $\capM{d}$.}
For an element $u$ of a free abelian group denote by $\di u$ the divisibility of $u$,
i.e. $\di 0=0$ and $\di u$ is the largest integer which divides $u$ for $u\ne0$.
For an element $u$ of an abelian group $G$ denote by $\di u$ the divisibility of $[u]\in G/\Tors(G)$.
%+\Tors(G); the minimal subgroup $kG$ of $G$ containing $u$
%For a subset $U$ of an abelian group denote $d(U):=\gcd\{d(u)\ :\ u\in U\}$.

Denote by $B(H_3)$ the space of bilinear forms $H_3\times H_3\to\Z$.
For $l\in B(H_3)$ denote by $\overline l:H_3\to H_1$ the adjoint homomorphism uniquely defined by the property
$l(x,y)=x\capM{N} \overline ly$.
A pair $(u,l)\in H_2\times B(H_3)$ is called {\it symmetric} if
$$l(y,x)=l(x,y)+u\capM{N} x\capM{N} y \quad\text{for all}\quad x,y\in H_3.$$

For $u\in H_2$, $l\in B(H_3)$ and $d:=\di u\in\Z$ denote
$$
K_{u,l}:=\ker(2\rho_d\overline l)\subset H_3\quad\text{and}\quad C_{u,l}:=\coker(2\rho_d\overline l).
$$
If the pair $(u,l)$ is symmetric, then a bilinear map
$$
\capM{d}\colon C_{u,l}\times K_{u,l}\to\Z_d\quad\text{is well-defined by}
\quad [x]\capM{d} y:=x\capM{N} y.
\footnote{\label{f:capd} Indeed, for each $x\in H_3$ and $y\in K_{u,l}$ we have
$2\overline lx\capM{N} y=2l(x,y)\underset d\equiv2l(y,x)=2\overline ly\capM{N} x\underset d\equiv0$.
Hence $\im(2\rho_d\overline l)\capM{N} K_{u,l}=\{0\}\subset\Z_d$.}
$$

The maps $\varkappa,\lambda,\beta_{u,l}$, and $\eta_{u,l,b},\theta_{u, l, b}$ of Theorem \ref{t:gensm} below are defined in \S\ref{s:defandplan-eta}.
The definitions of $\varkappa,\lambda,\beta_{u,l}$ are recalled from \cite[\S2.2, \S2.3]{CS16} and the definitions of $\eta_{u,l,b}$ and $\theta_{u, l, b}$ are new.

% respectively.
%\S\ref{s:defandplan-kl}, \S\ref{s:defandplan-b} .

\begin{Theorem}\label{t:gensm} Let $N$ be a closed connected orientable 4-manifold with torsion free $H_1$.

(a) The product
$$\varkappa \times\lambda\colon E^7(N)\to H_2^{DIFF}\times B(H_3)$$
has non-empty image consisting of symmetric pairs.

(b) For every $(u,l)\in \im (\varkappa\times\lambda)$ denote $d:=\di u$.
Every map
$$
\beta_{u,l}\colon (\varkappa\times\lambda)^{-1}(u,l)\to C_{u,l}
$$
is surjective \aronly{(see the remark immediately below the Theorem)}.

(c) For every $b\in C_{u,l}$ every map
$$\eta_{u,l,b}\colon \beta_{u,l}^{-1}(b)\to\frac{\Z_{\widehat d}}{\im\theta_{u,l,b}}$$
is an injection whose image consists of all even elements \aronly{(see the remark immediately below the Theorem)}.
Moreover, the map
$$
\theta_{u,l,b}\colon K_{u,l}\to4\Z_{\widehat d}
$$
is a homomorphism and
$$
\theta_{u,l,b}(y)-\theta_{u,l,b'}(y)=4\rho_{\widehat d}(b-b')\capM{d} y\quad\text{for every}\quad
b,b'\in C_{u,l}\quad \text{and}\quad y\in K_{u,l}.
$$

(d) $|\beta_{u,l}^{-1}(b)|=
%\dfrac{|2\Z_{\widehat d}|}{|\im \theta_{u,l,b}|}=
\dfrac{\widehat u}{\gcd(u,2)\cdot|\im \theta_{u,l,b}|}.$
\end{Theorem}

We call geometrically defined maps invariants.
In particular, the maps $\lambda$ and $\varkappa$ are invariants.

\aronly{\smallskip
{\bf Remark on relative invariants.}}
The maps $\beta_{u,l}$ and $\eta_{u, l, b}$ are {\em relative invariants}.
\jonly{See \cite[the Remark after Theorem 1.6]{CS16o}.}
\aronly{For $\eta_{u, l, b}$ this means that for $[f_0], [f_1] \in \beta_{u,l}^{-1}(b)$ there is an
invariant $([f_0], [f_1]) \mapsto \eta(f_0, f_1)$ (defined in \S\ref{s:defandplan-eta})
and that $\eta_{u, l,b}(f) := \eta(f,f')$ for a fixed choice of $[f']\in\beta_{u,l}^{-1}(b)$.
We suppress the choice of $[f']$ from the notation.
For $\beta_{u,l}$ the situation is similar and is discussed in the remark immediately following \cite[Theorem 1.3]{CS16}.

\smallskip
}

Parts (a) and (b) of Theorem \ref{t:gensm}
%Description of $\im(\varkappa\times\lambda)$ and the surjectivity of $\beta$
follow from \cite[Theorem 1.3]{CS16}.
The new part of Theorem \ref{t:gensm} is (c), which is proven in \S\ref{s:defandplan-eta}.
% definition of $\theta_{u,l,b}$ (\S\ref{s:defandplan-eta}) and proof of its properties.
Part (d) follows  because by (c) $\im\eta_{u,l,b}=2\Z_{\widehat d}/\im \theta_{u,l,b}$.

We remark that Theorem \ref{t:s1s3sm} is not an immediate corollary of Theorem \ref{t:gensm},
cf.~\cite[Remarks 2.20.a and 2.24]{CS16}.

\begin{Addendum}
%{Lemma}[Additivity]
\label{l:addeta}
In the notation of Theorem \ref{t:gensm}, for each $a\in\Z_{12}$, $(u,l)\in\im(\varkappa\times\lambda)$,
$b\in C_{u,l}$ and $f\in\beta_{u,l}^{-1}(b)$
%embeddings $g\colon S^4\to S^7$ and $f\colon N\to S^7$
$$\eta_{u,l,b}(f\#a)=\eta_{u,l,b}(f)+[2a]\in \frac{\Z_{\widehat d}}{\im \theta_{u,l,b}}.$$
%[\eta(g,\inc)]\in \frac{\Z_{\widehat{\di\varkappa(f)}}}{\im\eta_{\varkappa(f),\lambda(f),\beta(f,f')}},$$
%where $\inc\colon S^4\to S^7$ is the standard embedding and
%$f'$ comes from the definition of $\beta_{\varkappa(f),\lambda(f)}$.
\end{Addendum}

This follows from the definition of $\eta_{u,l,b}$ (\S\ref{s:defandplan-eta}) and \cite[Lemma 4.3.b]{CS16}.
%Lemma \ref{l:etaeven}.b.
%Note that if $\beta(f_0,f_1)=0$, then %$\eta(f_0\#g,f_1)-\eta(f_0,f_1)=[\eta(g,\inc)]\in\Z_{\overline{\di(u)}}/\im\eta_{u,l,b}.$

\begin{Corollary}\label{c:gen}
For each $(u,l)\in \im (\varkappa\times\lambda)$ let $d:=\di u$.
There is $f_{u,l}\in (\varkappa\times\lambda)^{-1}(u,l)$ such that for each $f\in (\varkappa\times\lambda)^{-1}(u,l)$ and $a,a'\in \Z_{12}$, denoting $b:=\beta(f,f_{u,l})\in C_{u,l}$ we have

(a) \quad $f\#a=f\#a'\quad\Leftrightarrow\quad a=a'$, \quad provided either

$\bullet$ $u=0$ and $\di b$ is divisible by 6, or

$\bullet$ $u\ne0$, $2\rho_d\overline l=0$ and $u$ is divisible by $24\ord(4b)$;

(b) \quad $f\#a=f\#a'\quad\Leftrightarrow\quad a\equiv a'\mod2\gcd(\di b,6),$ \quad provided $u=0$;
%$H_2=0$ (then $u=0$)}.$$
%\begin{cases}a=a' & \lambda(f)\ne0 \\
%a\equiv a'\mod2\gcd(\beta(f,f_0),6) &\lambda(f)=0 \end{cases}.$$

(c) \quad $f\#a=f\#a'\quad\Leftrightarrow\quad a\equiv a'\mod \dfrac{\widehat{\frac u{\ord(4b)}}}{\gcd(u,2)}$,\footnote{The class $u$ is divisible by $d$ and hence by the order $\ord(4b)$ of $d$ in the $d$-group $C_{u,l}$.} \quad
provided $u\ne0$ and $2\rho_d\overline l=0$.
\end{Corollary}

Part (a) follows from Parts (b,c).
% (for $u=0$ note that $\gcd(\di b,6)\ne!\gcd(b,6)$).
Parts (b,c) are proven in \S\ref{s:defandplan-s1s3d}.
Cf.~Remark \ref{r:cordif} and
\jonly{\cite[\S7 and Corollary 1.9]{CS16o}.}
\aronly{\S\ref{s:action} below.}

\aronly{
Theorem \ref{t:gensm} has the following restatement analogous to Theorem \ref{t:s1s3sm} and to
\cite[Corollary 2.13.b]{CS16}.
%Corollary \ref{t:rem}.b.

\begin{Corollary} \label{r:s1s3sm}
Denote by $B_0(H_3)$ the group of symmetric bilinear forms $H_3\times H_3\to\Z$.
Then there is a surjection
$$\tau:\Z_{12}\times H_1\times H_2^{DIFF}\times B_0(H_3)\to E^7(N)\quad\text{such that}$$
$$\tau(a,b,u,l)=\tau(a',b',u',l')\quad\Leftrightarrow\quad
u=u',\quad l=l',\quad b-b'\in K_{u,l_u}\quad\text{and}\quad a-a'\in\im\eta_{u,l_u,b},$$
where $l_u:=l+\lambda\tau(0,0,u,0)$.
\end{Corollary}
}

{\bf Acknowledgments.}
We would like to thank B.~Owens for assistance with the literature on $4$-manifolds.
We would like to thank the Hausdorff Institute for Mathematics and the University of Bonn for their hospitality and support during the early stages of this project.

\comment

To apply the method above to determine the action of knots

In \cite[Description of $d$-classes for $M_f$ Lemma 4.7]{CS16} we prove that $K_{u,l}$ acts freely on the set of $d$-classes.

For $f \in \beta_{u, l}^{-1}(b)$ we show that
\[ \im(\theta_{u, l, b}) = \{\eta(\varphi,Y) \, | \,  \de(W, z) = (M, Y), Y^2 = 0 \} \]

\begin{Corollary} \label{c:iss}
\remas{prove!} There are embeddings $f_0\colon S^1 \times S^3 \to S^7$ and $f_1 \colon S^2 \times S^2 \to S^7$
such that

$\bullet$ for every embedding $g:S^4\to S^7$ we have $[f_0\#g]=[f_0]$ and $[f_1\#g]=[f_1]$ but

$\bullet$ there is an embedding $g:S^4\to S^7$ such that $[f_0\#f_1\#g]\ne[f_0\#f_1]$.
\end{Corollary}

\begin{proof}
By Theorem \ref{t:s1s3sm}.a there are

$\bullet$ an embedding $f_0$ with $\lambda(f_0) = 4 \in B(H_3) = \Z$.

$\bullet$ an embedding $f_1$ with $\varkappa(f_1) = (24, 0) \in H_2(S^2 \times S^2) = \Z^2$.

We have $\varkappa(f_0\# f_1)=u:=(24, 0)$ and $\lambda(f_0\# f_1)=l:=8$ \remas{prove!}.
Then $\di u = 24$ and $K_{u,l} = 3H_3$.
By Theorem \ref{t:gensm}c we can find $b \in C_{u,l}$ with $\theta_{u, l, b} \neq 0$ \remas{prove!}.
Then $|\beta_{u, l}^{-1}(b)|>1$.
\end{proof}

\begin{Corollary} \label{ex:is}
\remas{prove!}
For $j = 2, 6$, there are embeddings $f_0^j \colon S^1 \times S^3 \to S^7$ and $f_1 \colon S^2 \times S^2 \to S^7$
such that $I(f_0^j) = 0 = I(f_1)$ but $I(f_0^j \# f_1) \cong \Z_{j}$.
\end{Corollary}

\begin{proof}
For $f_0^j$ take any embedding with $\lambda(f_0^j) = 2j \in B(H_3) = \Z$
and for $f_1$ take any embedding with $\varkappa(f_1) = (24, 0) \in H_2(S^2 \times S^2) = \Z^2$.
We have $(u, l) : = (\varkappa \times \lambda)(f_0^j \# f_1) = \bigl( (24, 0), 4j \bigr)$.
In this case $d = 24$ and $K_{u,l} = (j/2) \cdot H_3$ \remas{wrong for $j=2$}.
By Theorem \ref{t:gensm}c we can find $b \in C_{u,l}$ with $\theta_{u, l, b} \neq 0$ \remas{prove!}.
By Theorem \ref{t:gensm}c for every $f \in \beta_{u, l}^{-1}(b)$, $I(f) \cong j \cdot \Z_{12}$.
\end{proof}

%{\bf Plan of the paper.}

\endcomment

%\section{Definition of invariants and proofs modulo Lemma \ref{l:diff}} \label{s:defplan}

\section{Definition of the
%$\eta$- and $\theta$-
invariants}
\label{s:defandplan-eta}

In this paper we use conventions, notation and the following definitions of \cite[\S\S 2.1, 4.1]{CS16}.

$\bullet$ $N$ is a closed connected orientable $4$-manifold with torsion free $H_1$;

$\bullet$ $f,f_0,f_1 \colon N \to S^7$ are embeddings;

$\bullet$ $C=C_f$ is the closure of the complement in $S^7$ to a sufficiently small
tubular neighborhood of $f(N)$; the orientation on $C$ is inherited from the orientation of $S^7$;

 $\nus=\nus_f:\de C\to N$ is the sphere subbundle of the normal vector bundle of $f$:
the total space of $\nus$ is identified with $\partial C$.

%Denote by $\nud=\nud_f:S^7-\Int C_f\to N$ the oriented normal disk bundle of $f$
%(the orientation of $\nud$ is inherited from the orientation of $S^7$ and $N$).

Consider the following diagram:
$$
\xymatrix{& H_{q-2}(N) \ar[d]^{\nus^!} \ar@{=}[r]^{\text{PD}} \ar[dr]^{\widehat A} & H^{6-q}(N) \ar[d]^{AD}\\
H_{q+1}(C,\de) \ar[r]^{\de_C} \ar@{=}[d]^{\text{PD}} &H_q(\de C) \ar[d]^{\nus} \ar[r]^{i_C} &H_q(C) \\
H^{6-q}(C)  & H_q(N) \ar[l]^{AD} \ar[ul]^{A}}
$$
Here $AD$ is Alexander duality and $A=A_f,\widehat A=\widehat A_f$ are {\it homology Alexander duality} isomorphisms.

Define
$$\varkappa[f]:=A_f^{-1}(A_f[N]\capM{C_f}A_f[N])\in H_2$$
\cite[Lemma 3.2.$\varkappa'$]{CS16}. Define
$$\lambda[f](x,y):=x\capM{N}\widehat A_f^{-1}(A_f[N]\capM{C_f}\widehat A_fy)$$
for each $x,y\in H_3$ \cite[Lemma 3.2.$\lambda'$]{CS16}.

We abbreviate the subscript $f_k$ to just $k$.
For a bundle isomorphism $\varphi:\de C_0\to\de C_1$ define a closed oriented $7$-manifold $M=M_\varphi:=C_0\cup_\varphi(-C_1)$.
We call a bundle isomorphism $\varphi:\de C_0\to\de C_1$ a {\bf $\pi$-isomorphism} if $M_\varphi$ is parallelizable.
We shall omit the phrase `for a bundle isomorphism $\varphi$' if its choice is clear from the context.

If $P$ is a (compact oriented) codimension $c$ submanifold of a manifold $Q$ and either $y\in H_k(Q)$ or
$y\in H_k(Q,\de)$, denote
$$r_{P,Q}(y)=r_P(y)=y\cap P:=\text{PD}((\text{PD}y)|_P)\in H_{k-c}(P,\de).$$
A class $Y\in H_5(M_\varphi)$ is a {\bf joint Seifert class} if $Y\cap C_k=A_k[N]$ for each $k=0,1$ \cite[Lemma 3.13.a]{CS16}.
A joint Seifert class $Y\in H_5(M_\varphi)$ is called a {\bf $d$-class} for an integer $d$ if $\rho_dY^2=0$ (or, equivalently, $Y^2\in dH_3(M_\varphi)$).

Assume that $\varkappa(f_0)=\varkappa(f_1)$ and that $\lambda(f_0)=\lambda(f_1)$.
Denote $d:=\di\varkappa(f_0)$.
By \cite[Lemmas 2.4 and 2.5]{CS16}
%\ref{l:jhss} and \ref{l:piso}
there is a $\pi$-isomorphism $\varphi:\de C_0\to\de C_1$ and a joint Seifert class $Y\in H_5(M_\varphi)$.
Define
$$
\beta(f_0,f_1):=[(i_{\de C_0,M_\varphi}\nus_0^!)^{-1}\rho_dY^2]\in C_{\varkappa(f_0),\lambda(f_0)}
$$
using the composition
$H_1(N;\Z_d)\xrightarrow{~\nus_0^!~} H_3(\de C_0;\Z_d)\xrightarrow{~i_{\de C_0,M_\varphi}~} H_3(M_\varphi;\Z_d).$

\smallskip
{\bf Definition of $\eta(\varphi,Y)$ for a $\pi$-isomorphism $\varphi\colon\de C_0\to\de C_1$ and a $d$-class $Y\in H_5(M_\varphi)$, of $Y_{f,y}$ and $\theta(f,y)$.}
Since $\varphi\colon \de C_0\to\de C_1$ is a $\pi$-isomorphism, $M_\varphi$ is spin.
Take any normal spin structure on $M$ given by \cite[Lemma 4.2]{CS16}.
%\ref{l:krst}.
Since $M_\varphi$ is simply-connected, a normal spin structure on $M_\varphi$ is unique.
Since $\Omega_7^{Spin}(\C P^\infty)=0$ \cite[Lemma 6.1]{KS91} there is a 8-manifold $W$ with a normal
spin structure and $z\in H_6(W,\partial)$ such that $\partial W\underset{spin}=M_\varphi$ and $\partial z=Y$.
Consider the following fragment of the exact sequence of the pair $(W,\de W)$:
$$H_4(\de W;\Z_d)\overset{i_W}\rightarrow H_4(W;\Z_d) \overset{j_W}\rightarrow H_4(W,\de;\Z_d)
\overset{\de_W}\rightarrow H_3(\de W;\Z_d).$$
Since $\partial_W\rho_d z^2=\rho_dY^2=0$, there is a class $\overline{z^2}\in H_4(W;\Z_d)$ such that $j_W\overline{z^2}=\rho_dz^2$.
Denote by $p_W^*\in H_4(W,\partial)$ the spin characteristic class \cite[\S3.1]{CS16}.
Define
$$\eta(\varphi,Y)=\eta(f_0,f_1,d,\varphi,Y):=\rho_{\widehat d}(\overline{z^2}\capM{W}
%{\de}
(z^2-p_W^*))\in\Z_{\widehat d}.$$
For $y\in H_3$ denote
$$
Y_{f,y}:=\partial(A_f[N]\times I)+i\widehat A_fy\in H_5(C_f \times I)\quad\text{and}\quad
\theta(f,y):=\eta(\id\de C_f,Y_{f,y})\in\Z_{\widehat d}.
$$

\begin{Lemma}[proved in \S\ref{s:eta-prb},\S\ref{s:eta-prac}]\label{l:diff}
(a) For every $f$ and $y$, $\theta(f, y)$ is divisible by $4$.

(b) The map $\theta(f,\cdot)\colon K_{\varkappa(f),\lambda(f)}\lambda(f)\to\Z_{\widehat d}$ is a homomorphism, where $d:=\di\varkappa(f)$.

(c) For every $[f_0], [f_1] \in (\varkappa\times\lambda)^{-1}(u,l)$ and $y\in K_{u,l}$, we have for $d:=\di u$ that $\theta(f_0,y)-\theta(f_1,y)=4\rho_{\widehat d}(\beta(f_0,f_1)\capM{N} y)$.
\end{Lemma}

{\bf Definition of $\theta_{u,l,b}$.}
Take any $(u,l)\in \im (\varkappa\times\lambda)$ and $b\in C_{u,l}$. Let $d:=\di u$.
Define
$$\theta_{u,l,b}:K_{u,l}\to4\Z_{\widehat d}\quad\text{by}\quad\theta_{u,l,b}(y):=
\theta(f,y), \quad\text{where}\quad [f]\in\beta_{u,l}^{-1}(b).$$
The map $\theta_{u,l,b}$ is well-defined (i.e.~is independent of the choice of $f$) and is a homomorphism by Lemma \ref{l:diff}.ab and the transitivity of $\beta$ \cite[Lemma 2.10]{CS16}.
%(Lemma \ref{l:Dif} below).

\smallskip
{\bf Definition of $\eta(f_0,f_1)$.}
Take representatives $f_0,f_1$ of two isotopy classes in $(\varkappa\times\lambda)^{-1}(u,l)$ such that $\beta(f_0,f_1)=0$.
By \cite[Lemma 2.5]{CS16}
%Lemma \ref{l:piso}
there is a $\pi$-isomorphism $\varphi\colon \de C_0\to\de C_1$.
By \cite[Lemma 4.1]{CS16}
%Lemma \ref{l:fc}
there is a $d$-class $Y\in H_5(M_\varphi)$ for $d:=\di\varkappa(f_0)$.
Define
$$
\eta(f_0,f_1):=[\eta(\varphi,Y)]\in\frac{\Z_{\widehat d}}{\im\theta_{u,l,b}}.
$$
This is well-defined by \cite[Lemma 4.3.c]{CS16}
%Lemmas \ref{l:etaeven}.c
and Lemma \ref{l:welleta}.a below, and is even by \cite[Lemma 4.3.a]{CS16}.
%Lemma \ref{l:etaeven}.a, all lemmas stated below.

\begin{Lemma}\label{l:traeta} Let $f_0,f_1,f_2\colon N\to S^7$  be embeddings and
$\varphi_{01}\colon \de C_0\to\de C_1$, $\varphi_{12}\colon \de C_1\to\de C_2$
$\pi$-isomorphisms and $Y_{01}\in H_5(M_{\varphi_{01}})$, $Y_{12}\in H_5(M_{\varphi_{12}})$ \ $d$-classes.
% for an integer $d$.
Then $\varphi_{02}\colon =\varphi_{12}\varphi_{01}$ is a $\pi$-isomorphism and there is a $d$-class
$Y_{02}\in H_5(M_{\varphi_{02}})$ such that $\eta(\varphi_{02},Y_{02})=\eta(\varphi_{01},Y_{01})+\eta(\varphi_{02},Y_{12})$.
\end{Lemma}

%\begin{proof}Give details???\end{proof}
This is proved analogously to \cite[Lemma 2.10]{CS11}, cf.~\cite[\S2, Additivity Lemma]{Sk08'}
(the property that $Y_{02}$ is a $d$-class is achieved analogously to \cite[\S4.3, proof of Lemma 4.6]{CS16}).
%\ref{l:cobordeta}).

\begin{Lemma}\label{l:welleta}
%Take representatives $f_0,f_1$ of two isotopy classes in $(\varkappa\times\lambda)^{-1}(u,l)$ such that %$\beta(f_0,f_1)=0$.
Let $[f_0], [f_1] \in (\varkappa\times\lambda)^{-1}(u,l)$ be such that $\beta(f_0,f_1)=0$.
Denote $d:=\di u$.
Take any $\pi$-isomorphism $\varphi\colon \de C_0\to\de C_1$.

(a) The residue $\eta(f_0,f_1)$ is independent of the choice of a $d$-class $Y\in H_5(M_\varphi)$.

(b) If $\eta(f_0,f_1)=0$, then there is a $d$-class $Y\in H_5(M_\varphi)$ such that
$\eta(\varphi,Y)=0\in\Z_{\widehat d}$.
%and, if $d$ is even, $\eta'(M_\varphi,Y)=0$ (see definition in \S\ref{s:defandplan-eta}).
\end{Lemma}

\begin{proof}[Proof of (a)]
Take any pair of $d$-classes $Y',Y''\in H_5(M_\varphi)$.
Part (a) follows because
$$
\eta(\varphi,Y')-\eta(\varphi,Y'')\overset{(1)}= \eta(\id\de C_0,Y)\overset{(2)}= \theta(f_0,y)=\theta_{u,l,\beta(f_0,f')}(y)\in\Z_{\widehat d},
$$
where

$\bullet$ equality (1) holds for some $d$-class $Y\in H_5(M_{f_0})$ by Lemma \ref{l:traeta};
% below;

$\bullet$ equality (2) holds for some $y\in K_{u,l}$ by the description of $d$-classes \cite[Lemma 4.7]{CS16}.
%(Lemma \ref{l:simpleY0} below).
\end{proof}

\begin{proof}[Proof of (b)]
%Take any $\pi$-isomorphism $\varphi\colon \de C_0\to\de C_1$.
Part (b) follows because
$$0\overset{(1)}=\eta(\varphi,Y')-\theta_{u,l,\beta_{u,l}(f_0)}(y)=\eta(\varphi,Y')-\theta(f_0,y)\overset{(3)}= \eta(\varphi,Y)\in\Z_{\widehat d},$$
where

$\bullet$ equality (1) holds for some $d$-class $Y'\in H_5(M_\varphi)$ and $y\in K_{u,l}$ because $\eta(f_0,f_1)=0$;

$\bullet$ equality (3) holds for some $d$-class $Y\in H_5(M_\varphi)$ by Lemma \ref{l:traeta}.
%below
\end{proof}

\begin{Lemma}[Transitivity of $\eta$]\label{l:difeta}
For any triple of embeddings $f_0,f_1,f_2\colon N\to S^7$ having the same values of $\varkappa$- and $\lambda$-invariants and the property that that $\beta(f_0,f_1)=\beta(f_1,f_2)=0$,
we have $\eta(f_2,f_0)=\eta(f_2,f_1)+\eta(f_1,f_0)$.
\end{Lemma}

This follows by Lemma \ref{l:traeta}.

\begin{Theorem}[Isotopy classification]\label{l:Is}
If $\lambda(f_0)=\lambda(f_1)$, $\varkappa(f_0)=\varkappa (f_1)$, $\beta(f_0,f_1)=0$ and
$\eta(f_0,f_1)=0$, then $f_0$ is isotopic to $f_1$.
\end{Theorem}

\begin{proof}
The proof is analogous to the proof of \cite[Isotopy Classification Modulo Knots Theorem 2.8]{CS16}.
%\ref{l:isomk}
We only need to replace the second paragraph of that proof by the following sentence: `Since $\eta(f_0,f_1)=0$, by Lemma \ref{l:welleta}.b we can change $Y$ and assume additionally that $\eta(\varphi,Y)=0$.'
\end{proof}

{\bf Definition of $\eta_{u,l,b}$.}
Take any $[f_0]\in \beta_{u,l}^{-1}(b)$.
Define the map
$$
\eta_{u,l,b}\colon \beta_{u,l}^{-1}(b)\to\frac{\Z_{\widehat d}}{\im\theta_{u,l,b}}\quad\text{by}
\quad \eta_{u,l,b}[f]:=\eta(f,f_0).
$$
The map $\eta_{u,l,b}$ depends on $f_0$ but we do not indicate this in the notation.

\begin{proof}[Proof of Theorem \ref{t:gensm}.c]
The property on $\theta_{u,l,b}-\theta_{u,l,b'}$ holds by Lemma \ref{l:diff}.c.
The map $\eta_{u,l,b}$ is injective by the Isotopy Classification Theorem \ref{l:Is}.
The image of this map consists of all even elements by \cite[Lemma 4.3.a]{CS16}
%Lemma \ref{l:etaeven}.a
and Addendum \ref{l:addeta}.
%the additivity of $\eta$ (Lemma \ref{l:addeta}).}
\end{proof}

\comment

We also use the following lemmas proved in \cite{CS16}.

\begin{Lemma}[Transitivity of $\beta$]\label{l:Dif}
For every triple of embeddings $f_0,f_1,f_2\colon N\to S^7$ having the same values of $\varkappa$- and $\lambda$-invariants
we have $\beta(f_2,f_0)=\beta(f_2,f_1)+\beta(f_1,f_0)$.
\cite[Lemma 2.10]{CS16}
\end{Lemma}

\begin{Lemma}\label{l:piso}
If $\varkappa(f_0)=\varkappa(f_1)$ and $\lambda(f_0)=\lambda(f_1)$, then there is a $\pi$-isomorphism
$\varphi\colon \de C_0\to\de C_1$.
A $\pi$-isomorphism is unique (up to vertical homotopy through linear isomorphisms) over $N_0$.
\cite[Lemma 2.5]{CS16}
\end{Lemma}

\begin{Lemma} \label{l:fc}
If $\lambda(f_0)=\lambda(f_1)$,\ $\varkappa(f_0)=\varkappa(f_1)$,\ $\beta(f_0,f_1)=0$ and
$\varphi\colon \de C_0\to\de C_1$ is a $\pi$-isomorphism, then there is a $\di(\varkappa(f_0))$-class for $\varphi$.
\cite[Lemma 4.1]{CS16}
\end{Lemma}

\begin{Lemma}[Description of $d$-classes for $M_f$] \label{l:simpleY0}
A class $Y\in H_5(M_f)$ is a $d$-class if and only if $Y=Y_{f,y}$ for some
$y\in K_{\varkappa(f),\lambda(f)}$. \cite[Lemma 4.7]{CS16}
\end{Lemma}

\begin{Lemma}\label{l:etaeven} \cite[Lemma 4.3]{CS16}
Let $\varphi\colon \de C_0\to \de C_1$ be a $\pi$-isomorphism and $Y\in H_5(M_\varphi)$ a $d_0$-class.

(a) (Divisibility of $\eta$ by 2) The residue $\eta(\varphi,Y)\in\Z_{\widehat{d_0}}$ is even.

(b) (Change of $\eta$)
There is an embedding $g\colon S^4\to S^7$, a $\pi$-isomorphism $\varphi'\colon \de C_0\to \de C_{f_1\#g}$ and a $d_0$-class
$Y'\in H_5(M_{\varphi'})$ such that $\eta(\varphi',Y',f_0,f_1\#g,d_0)=\eta(\varphi,Y,f_0,f_1,d_0)+2$.

(c)(Change of $\varphi$)
For every $\pi$-isomorphism $\varphi'\colon \de C_0\to \de C_1$ there is a $d_0$-class $Y'\in H_5(M_{\varphi'})$ such that $\eta(\varphi',Y')=\eta(\varphi,Y)$.
\end{Lemma}

\begin{Remark}\label{r:inman}
Recall that the inertia group $I(P)$ of a closed oriented $n$-manifold $P$ is defined to the subgroup
of homotopy spheres $\Sigma \in \Theta_n$ such that there is an orientation preserving diffeomorphism
$P \cong P \# \Sigma$.
In particular, our Corollary \ref{ex:is} can be compared with \cite[\S6]{Wi74} where examples of
$7$-manifolds $P_0$ and $P_1$ are given with $I(P_0) = I(P_1) = 0$, but $I(P_0 \# P_1) \ne 0$.
%\subset \Theta_7 \cong \Z_{28}$.
\end{Remark}

{\bf Old version of proof of Lemma \ref{l:unimzd}.}

\begin{proof}
Let us prove the `$\subset$' part.
Assume that $q$ is a divisor of $c+\Tors\coker m$.
Then there exist $s,l_1,\ldots,l_s\in\Z$ and $c_0,t_1,\ldots,t_s\in V^*$ such that $c=qc_0+t_1+\ldots+t_s$ and $l_nt_n\in \im m$ for every $n=1,\ldots,s$.
Take any $y\in\ker m$.
Since the polarization of $m$ is symmetric, we have $l_nt_n(y)=0\in\Z$, thus $t_n(y)=0$.
Hence $c(y)=qc_0(y)+(t_1+\ldots+t_s)(y)=qc_0(y)$  is divisible by $q$.

Let us prove the `$\supset$' part.
Assume that $q$ is a divisor of $c|_{\ker m}$.
The subgroup $\ker m\subset V$ is a direct summand.
Take a decomposition $V=\ker m\oplus V'$.
%If $c|_{\ker m}$ is divisible by $q$, then $c=qc_0\oplus c|_{V'}$.
Since $m|_{V'}\colon V'\to\im m$ is an isomorphism, there is an element $x\in V'$ such that $c|_{V'}=m(x)|_{V'}$.
Since the polarization of $m$ is symmetric, $m(x)|_{\ker m}=0$.
Then  $c-m(x)$ coincides with $c$ on $\ker m$ and is zero on $V'$.
So $c-m(x)=qc_0$ for some $c_0\in V^*$.
Hence $d(c+\Tors\coker m)$ is divisible by $q$.
\end{proof}

\endcomment

\section{Proof of Corollaries \ref{t:fa} and \ref{c:gen}.bc}\label{s:defandplan-s1s3d}

\begin{proof}[Proof of Corollary \ref{t:fa}]
By Theorem \ref{t:gensm}.c and Addendum \ref{l:addeta} $(ii)\Rightarrow (i)$.

The other direction is implied by the following assertions.

(*) {\it If the action of knots is free and $H_1$ is torsion free, then $H_1=0$.}

(**) {\it If the action of knots is free and $H_1=0$, then $H_2=0$.}

%(*) {\it If $H_1\ne0$, then there is an embedding $f \colon N \to S^7$ such that
%$K_{\varkappa(f),\lambda(f)}$ contains a primitive element of $H_3$.}
%$\ker\overline{\lambda(f)}$ contains a non-zero direct summand.}

%(**) {\it If the action of knots is free, then for each embedding $f \colon N \to S^7$ the kernel
%$K_{\varkappa(f),\lambda(f)}$ does not contain a primitive element of $H_3$.}
%non-zero direct summand.}

\smallskip
{\it Proof of (*).}
By Theorem \ref{t:gensm}.a there is an embedding $f_1:N\to S^7$ and
$(u, l_1) : = (\varkappa \times \lambda)(f_1)$ is a symmetric pair.
If $H_1\ne0$, then there is a basis $\{y_1, \ldots, y_n\}$ for $H_3$ with $n>0$.
Express $l_1 = l_1^{ij}$ as a matrix with respect to this basis.
For any symmetric matrix $a^{ij}$ the pair $(u, l_1- a)$ is again a symmetric pair.
Take $a^{ij}$ to be the symmetric matrix with $a^{ij} = l_1^{ij}$ for $i \leq j$.
Then $l: = l_1 - a$ is strictly upper-triangular with respect to the chosen basis; i.e.~$l^{ij} = 0$ for $i\ge j$.
The pair $(u, l)$ is symmetric and $l(y_1) = 0$.
By Theorem \ref{t:gensm}.a there is an embedding $f$ with $(\varkappa \times\lambda)(f) = (u,l)$.

Since $y_1\in K_{u,l}$ is primitive, by Poincar\'{e} duality there is $x\in H_1$ such that $x\capM{N} y_1 = 1$.
Since the action of knots is free, by Theorem \ref{t:gensm}.c and Addendum \ref{l:addeta}
$u$ is divisible by $24$ and $\theta_{u, l, b}=0$ for every $b\in C_{u,l}$.
Then $d:=\di u$ is divisible by $24$, and so is $\widehat d$.
Hence by Theorem \ref{t:gensm}.c
$$
0-0 = \theta_{u, l,b+[x]}(y_1)-\theta_{u, l,b}(y_1) = 4\rho_{\widehat d}(x\capM{N}y_1) = 4\ne0 \in\Z_{\widehat d}.
$$
This contradiction shows that $H_1=0$.

\smallskip
{\it Proof of (**).}
Since the action of knots is free, by Theorem \ref{t:gensm}.c and Addendum \ref{l:addeta}
every element $u\in H_2^{DIFF}$ is divisible by $24$.
Since $\rho_2u = {\rm PD}w_2(N)$, we obtain $w_2(N) = 0$, so the intersection form $\capM{N}$ of $N$ is even.
If $H_2\ne0$, then the intersection form of $N$ is indefinite \cite[Theorem 1]{Do87}.
Hence by \cite[Theorem 1.2.21]{GS99}
%\cite[Ch.\,II Theorem 5.3]{MH73}
this form is isomorphic to $m H_+\oplus n E_8$, where $H_+$ is the standard hyperbolic form with matrix $\left(\begin{matrix}0&1\\1&0\end{matrix}\right)$, and the form $E_8$ is positive definite, so $m>0$.
By \cite[Lemma 1.2.20]{GS99} $\sigma(N)$ is divisible by 4.
Then $u:=(2,\sigma(N)/2)\oplus(m-1)0\oplus n0\in H_2^{DIFF}$ and $u$ is not divisible by $24$.
This contradiction shows that $H_2=0$.
%Then $H_2=0$ by \cite[Theorem 1]{Fu01}.
\end{proof}

\begin{Remark} \label{r:cordif}
In Corollary \ref{c:gen}.bc each of the assumptions `$u=0$' and `$u\ne0$ and $2\rho_d\overline l=0$'
can be replaced by each of the following successively weaker assumptions:

(1) $\rho_{\widehat d}K_{u,l}\subset\rho_{\widehat d}H_3$ is a direct summand, or

(2) every homomorphism $\rho_{\widehat d}K_{u,l}\to4\Z_{\widehat d}$ extends to $\rho_{\widehat d}H_3$, or

(3) there is an element $\t b\in C_{u,l}$ such that $\theta_{u,l,\t b}=0$.
\end{Remark}

Clearly,
$`\text{either }u=0\text{ or }2\rho_d\overline l=0'\ \Rightarrow (1) \Rightarrow(2).$

%Since $H_3$ is free, $K$ is a free $\Z_d$-module. Let $y_1,\dots,y_s$ be free generators of $K$.
%Define a homomorphism $\t\eta\colon K\to 4\Z_d$ by letting $\t\eta(y_n)$ be a lift of $\eta_{u,l,b'}(y_n)$ for
%each $n=1,2,\dots,s$. Then $\eta_{u,l,b'}=4\rho_{24}\t\eta$.

\begin{proof}[Proof that $(2)\Rightarrow(3)$]
Take any $b'\in C_{u,l}$.
We have $\theta_{u,l,b'}=\theta_{u,l,b'}^+\rho_{\widehat d}$ for some homomorphism
$\theta_{u,l,b'}^+\colon \rho_{\widehat d}K_{u,l}\to 4\Z_{\widehat d}$.
%Since $\rho_{\widehat d}K\subset \rho_{\widehat d}H_3$ is a direct summand,
Extend $\theta_{u,l,b'}^+$ to a homomorphism $\rho_{\widehat d}H_3\to 4\Z_{\widehat d}$.
Since $H_3$ is free, $\rho_{\widehat d}H_3$ is a free $\Z_{\widehat d}$-module.
Hence the latter homomorphism is divisible by 4.
Then by Poincar\'e duality there is a class $x\in\rho_{\widehat d}H_1$ such that
$\theta_{u,l,b'}^+(z)=4x\capM{N}z$ for every $z\in \rho_{\widehat d}K_{u,l}$.
Let $\t b:=b'+[\t x]$, where $\t x\in \rho_dH_1$ is a lifting of $x$.
Then by Theorem \ref{t:gensm}
$$\theta_{u,l,\t b}(y)=\theta_{u,l,b'}(y)-4\rho_{\widehat d}([\t x]\capM{N} y)=
\theta_{u,l,b'}^+(\rho_{\widehat d}y)-4x\capM{N}\rho_{\widehat d}y =0
\quad\text{for every}\quad y\in K_{u,l}.$$
\end{proof}

\begin{proof}[Proof of Corollary \ref{c:gen}.bc under the assumption (3) of Remark \ref{r:cordif}]
Define the element $\beta_{u,l}'(f):=\t b-\beta_{u,l}(f)\in C_{u,l}$.
Then $\theta'_{u,l,b}=\theta_{u,l,\t b-b}$ for every $b\in C_{u,l}$, hence $\theta'_{u,l,0}=0$.
Therefore we may assume that $\beta_{u,l}$ is chosen so that $\theta_{u,l,0}=0$.

Take any $b\in C_{u,l}$ and denote $K:=4b\capM{d}K_{u,l}\subset\Z_d$.
So
$$
\gcd(d,2)\cdot|\beta_{u,l}^{-1}(b)|\overset{(1)}=\frac{\widehat d}{|\im\theta_{u,l,b}|}\overset{(2)}=
\frac{\widehat d}{|\rho_{\widehat d}K|}=[\Z_{\widehat d}:\rho_{\widehat d}K]=
\gcd(\widehat d,[\Z_d:K])=\widehat{[\Z_d:K]},
$$
where equalities (1) and (2) hold by Theorem \ref{t:gensm}.d.
%$\bullet$ equality (2) holds because  $\eta_{u,l,b}(y)=\rho_{\widehat d}(4b\cap y)$ for each $y\in K_{u,l}$, and
%$\bullet$ equality (4) holds because  $G=[\Z_q:G]\Z_q$ for a subgroup $G\subset\Z_q$. ?q=G=0?
Now Corollary \ref{c:gen}.bc is implied by Addendum \ref{l:addeta} and the following Lemma \ref{l:unimzd}.
\end{proof}

\begin{Lemma}\label{l:unimzd} Let $V$ be a free $\Z$-module, $d$ an integer, $\rho_d:V\to V/dV$ the reduction mod $d$ and $m\colon V\to V^*$ a homomorphism whose polarization $V\times V\to\Z$ has a symmetric mod $d$ reduction.
Then

(a) a bilinear map $\capM{d}:\coker(\rho_dm)\times \ker(\rho_dm)\to\Z_d$ is well-defined by $[x]\capM{d}y:=\rho_dx(y)$ for $x\in V^*$.   

(b) for every $c\in\coker(\rho_d m)$
$$[\Z_d:c\capM{d}\ker(\rho_dm)]=
\begin{cases} \di c &d=0,\\
\dfrac d{\ord c} &d\ne0.\end{cases}$$
\end{Lemma}

%Then $c\cdot K=\begin{cases} d(c+K^\perp+\Tors(\Z^n/K^\perp))\Z &d=0\\
%\dfrac d{\ord\phantom{}_{\Z_d^n/K^\perp}(c+K^\perp)}\Z_d &d\ne0\end{cases}.$

This lemma is elementary and so possibly known.
Part (a) is simple and is essentially proved in footnote \ref{f:capd}.
\jonly{Part (b) is proved in \cite[\S3]{CS16o}.}

\aronly{

\begin{proof}[Proof of part (b) for $d=0$]
We need to prove the following:

{\it Let $V$ be a free $\Z$-module and $m\colon V\to V^*$ a homomorphism whose polarization $V\times V\to\Z$ is symmetric.
Then for every $c\in\coker m$
%(old version: $c\in V^*$)
we have $[\Z:c(\ker m)]=\di c$.}

Assume that $q$ is a divisor of $c+\Tors\coker m$.
Then there exist $s\in\Z$, $l_1,\ldots,l_s\in\Z-\{0\}$ and $c_0,t_1,\ldots,t_s\in V^*$ such that $c=qc_0+t_1+\ldots+t_s$ and $l_nt_n\in \im m$ for every $n=1,\ldots,s$.
Take any $y\in\ker m$.
Since the polarization of $m$ is symmetric, we have $l_nt_n(y)=0\in\Z$, thus $t_n(y)=0$.
Hence $c(y)=qc_0(y)+(t_1+\ldots+t_s)(y)=qc_0(y)$ is divisible by $q$.
Thus $[\Z:c(\ker m)]$ is divisible by $q$.

Assume that $q$ is a divisor of $[\Z:c(\ker m)]$.
Then $c|_{\ker m}$ is divisible by $q$.
The subgroup $\ker m\subset V$ is a direct summand.
Take a decomposition $V=\ker m\oplus V'$.
 Since $m|_{V'}\colon V'\to\im m$ is an isomorphism, there is an element $x\in V'$ such that $c|_{V'}=m(x)|_{V'}$.
Since the polarization of $m$ is symmetric, $m(x)|_{\ker m}=0$.
Then  $c-m(x)$ coincides with $c$ on $\ker m$ and is zero on $V'$.
So $c-m(x)=qc_0$ for some $c_0\in V^*$.
Thus $c+\Tors\coker m$ is divisible by $q$.
\end{proof}

\begin{proof}[Proof of part (b) for $d\ne0$]
We need to prove that $|c(\ker(\rho_dm))|=\ord c$ for every
%$d\ne0$ and
$c\in \coker(\rho_dm)$.
(We remark that this is obvious for $\rho_dm=0$ and this case is sufficient for Corollary \ref{c:gen}.c.)

Denote $K:=\rho_d\ker(\rho_dm)\subset V/dV$.
Since the polarization $V\times V\to\Z$ of $m$ has a symmetric mod $d$ reduction,
$\im(\rho_d m)\subset K^\perp\subset V^*/dV^*$.
Since $|\im(\rho_d m)|=\dfrac{|V/dV|}{|K|}=|K^\perp|$, it follows that $\im(\rho_d m)=K^\perp$.
Now the required assertion follows because for every $c'\in V^*/dV^*$
%$d\ne0$, $c\in\Z_d^n$, subgroup $K\subset\Z_d^n$ and the standard scalar product $\Z_d^n\times \Z_d^n\to\Z_d$ we have
$$|c'(K)|=\frac d{\di c'(K)}= \min\{r\ |\ rc'(K)=0\}=
\min\{r\ |\ rc'\in K^\perp\}=\ord\phantom{}_{(V^*/dV^*)/K^\perp}(c'+K^\perp).$$
\end{proof}

}

%\section{Proof of Lemma \ref{l:diff} on $\theta$-invariant}\label{s:eta-pr}

%\newpage
\section{Proof of Lemma \ref{l:diff}.b}\label{s:eta-prb}

Before reading the proof of Lemma \ref{l:diff} we recommend reading the idea of the proof in
\aronly{\S\ref{s:idea}.}
\jonly{\cite[\S6]{CS16o}.}

In this and the following section
%we use homological Alexander duality and the restriction homomorphism defined in \cite[\S3.1]{CS16};
%$f,f_0,f_1\colon N\to S^7$ are embeddings, representing elements of $(\varkappa\times\lambda)^{-1}(u,l)$;
$l=\lambda(f)=\lambda(f_0)=\lambda(f_1)$, $u=\varkappa(f)=\varkappa(f_0)=\varkappa(f_1)$ and $d=\di u$.
Denote $1_m:=(1,0,\ldots,0)\in S^m$, $\Delta:=1_2\times D^4\times1_1$ and $t:=S^2\times0\times S^2$.
For every $y\in H_3$ take the following objects constructed in \cite[Proof of Lemma 4.8]{CS16}:
6-manifolds $V\subset C_f$ and  $\widehat V:=V\cup_{S^2\times S^3}(S^2\times D^4\times1)$, an embedding
$v_2:S^2\times S^3\times D^2\to\Int C_f=\Int C_f\times\frac12$, 8-manifolds $W_-\subset C_f\times I$ and
$W:=W_-\cup_{S^2\times S^3\times S^2}(S^2\times D^4\times S^2)$, classes $Z\in H_6(W,\partial)$ and
$z:=Z+[\widehat V]\in H_6(W,\partial)$.
The objects are not uniquely constructed from $y$, and we allow arbitrary choices in that construction.
% \ref{l:prdiffbor}

\smallskip
{\bf Definition of $W',W'_-$ and $i'\colon W'\to W$.}
Let
$$W'_-:= C_f-\Int\im v_2 \quad \text{and} \quad
W':=W'_-\cup_{v_2|_{S^2\times S^3\times S^1}}S^2\times D^4\times S^1$$
%$$W':=W'_-\cup_{\widehat\sigma}S^2\times D^4\times S^1.$$
(the manifold $W'$ may be called the result of an $S^1$-parametric surgery along $v_2$.)
Define an embedding $W'_-\to W_-$ by $x\mapsto x\times1/2$.
We assume that this embedding and the standard embedding $S^2\times D^4\times S^1\to S^2\times D^4\times S^2$
(that is the product of the identity and the equatorial inclusion $S^1\to S^2$) fit together to give an embedding $$i'\colon W'\to W.$$
Observe that $\Delta,\widehat V\subset W'$.

\begin{Lemma}\label{l:resa} For every $y\in H_3$ we have
$$
z^2\capM{W} W_-\underset d\equiv2i_{V,W_-}(Z\cap V)\in H_4(W_-,\de)
$$
(since $\de V\subset\de W_-$, the inclusion induces a map $i_{V,W_-}\colon H_4(V,\de)\to H_4(W_-,\de)$).
%$[\widehat V]^2=0\in H_4(W,\de)$ and $\rho_dZ^2\cap W_-=0\in H_4(W_-,\de;\Z_d)$.
\end{Lemma}

\begin{proof}
Since $\widehat V\subset W'$, we have $[\widehat V]^2=0\in H_4(W,\de)$.
Also
$$Z^2\cap W_-=(A_f[N]\times I)^2\cap W_-=A_f\varkappa(f)\times I\cap W_-\underset d\equiv0\in H_4(W_-,\de).$$
Hence
$$z^2\cap W_-=(Z+[\widehat V])^2\cap W_- \underset d\equiv 2(Z\capM{W}[\widehat V])\cap W_-=
2i_{V,W_-}(Z\cap\widehat V\cap W_-)=2i_{V,W_-}(Z\cap V).$$
\end{proof}

\begin{proof}[Proof of Lemma \ref{l:diff}.b]
In this proof a statement or a construction involving $k$ holds or is made for $k=0,1$.
Given $y_k\in K_{u,l}$ construct the manifold $W_k$ as $W$ of \cite[Proof of Lemma 4.8]{CS16}
%\ref{l:prdiffbor}
by parametric surgery in $C_f\times[k-1,k]$.
We add the subscript $k$ to $W_-,W'_-,t,\Delta,Z,\widehat V,z$ constructed in \cite[Proof of Lemma 4.8]{CS16}.
(So unlike in other parts of this paper, a subscript 0 for a manifold does not mean deletion
of a codimension 0 ball from the manifold.)
Define
$$
W := W_0\cup_{C_f\times0}W_1\quad\text{and}\quad
W_-:=C_f\times[-1,1]-\Int\im(v_{3,0}\sqcup v_{3,1})=W_{0-}\cup_{C_f\times0}W_{1-}.
$$
The manifold $W$ just defined should not be confused with the manifolds which
were previously denoted $W$ but are now denoted $W_0$ and $W_1$.
The same remark holds for $W_-$ and for $Z,V,\widehat V,z$ constructed below.

The spin structure on $W_-$ coming from $S^7\times[-1,1]$ extends to $W$.
Clearly, we have
\linebreak
$\partial W\underset{spin}=\de(C_f\times[-1,1])\underset{spin}\cong M_f$
(for the `boundary' spin structures on $\de(C_f\times[-1,1])$ and on $M_f$).

Since $H_5(t_k\times\Delta_k)=0$, by the cohomological exact sequence of the pair $(W,W_-)$
(cf.~diagram (*) in \cite[Proof of Lemma 4.8]{CS16}), $r_{W_-}\colon H_6(W,\de)\to H_6(W_-,\de)$ is an epimorphism.
Take any
$$
Z\in r_{W_-}^{-1}(A_f[N]\times[-1,1]\cap W_-)\subset H_6(W,\de).
$$
Denote
$$
V:=V_0\sqcup V_1,\quad \widehat V:=\widehat V_0\sqcup \widehat V_1\quad\text{and}\quad z:=Z+[\widehat V]\in H_6(W,\de).
$$
Since $\de_WZ=Y_{f,0}$ and $\de_W[\widehat V_k]=i_{\de W}\widehat A_fy_k$, we have $\de z=Y_{f,y_0+y_1}$.
Thus the pair $(W,z)$ is a spin null-bordism of $(M_f,Y_{f,y_0+y_1})$.

Since $y_k\in K_{u,l}$, we have $\de z_k^2\underset d\equiv0$.
Take any $\overline{z_k^2}\in j_{W_k}^{-1}\rho_dz_k^2$.
Let
$$
\overline{z^2}:=i_{W_0,W}\overline{z_0^2}+i_{W_1,W}\overline{z_1^2}.
$$
Then $\overline{z^2}\cap W_k=\overline{z_k^2}$.
Also
$$
j_W\overline{z^2}\cap W_- = \sum\limits_{k=0}^1j_{W_k}\overline{z_k^2}\cap W_- =
\rho_d\sum\limits_{k=0}^1z_k^2\cap W_{k-}\quad\text{and}
$$
$$
\sum\limits_{k=0}^1z_k^2\cap W_{k-}\overset{(1)}\equiv 2\sum\limits_{k=0}^1i_{V_k,W_{k-}}(Z_k\cap V_k) =
2i_{V,W_-}(Z\cap V) \overset{(3)}\equiv z^2\cap W_-.
$$
Here the congruences (1) and (3) modulo $d$ hold by Lemma \ref{l:resa} and analogously to Lemma \ref{l:resa}, respectively.

Hence by the cohomological exact sequence of the pair $(W,W_-)$ with coefficients $\Z_d$
(cf.~diagram (*) in \cite[Proof of Lemma 4.8]{CS16})
$j_W\overline{z^2}-\rho_dz^2=n_0[t_0]+n_1[t_1]$ for some $n_0,n_1\in\Z_d$.
We have
$$
n_k[t_k]=(j_W\overline{z^2}-\rho_dz^2)\cap W_k=j_{W_k}\overline{z_k^2}-\rho_dz_k^2 =0\in H_4(W_k,\de;\Z_d).
$$
Therefore $n_0=n_1=0$.
So $j_W\overline{z^2}=\rho_dz^2$.

Since $\t{W_k}:=W_k-C_f\times[0,(2k-1)/3)$ is a deformation retract of $W_k$,
the inclusion $\t{W_k}\to W_k$ induces an isomorphism on $H_4$.
Clearly, $z\cap W_k=z_k$, so $z^2\cap W_k=z_k^2$.
Hence
$$
\overline{z^2}\capM{W}(z^2-p_W^*)=\sum\limits_{k=0}^1 (\overline{z^2}\cap W_k)\capM{W_k}((z^2-p_W^*)\cap W_k)=
\sum\limits_{k=0}^1 \overline{z_k^2}\capM{W_k}(z_k^2-p_{W_k}^*).
$$
So $\eta(f,\cdot)$ is a homomorphism.
\end{proof}

%\newpage
\section{Proof of Lemma \ref{l:diff}.ac}\label{s:eta-prac}
%$W'_-\to C_{f\sqcup v}$?

\begin{Lemma}\label{l:res} For every $y\in H_3$  we have:

(a) $\de(Z\cap V)=[\de\Delta]-i_{\partial C_f,\de V}\xi y\in H_3(\de V)$, where $\xi\colon N_0\to \partial C_f$
is  a weakly unlinked section for $f$ (see definition in \cite[\S2.2]{CS16});
%\S\ref{s:defandplan-kl}).

%(b) $Z\cap V=A_f[N]\cap V\in H_4(V,\de)$.
%(a) $\de(A_f[N]\cap V)=[\de\Delta]-i_{\partial C_f,\de V}\xi y\in H_3(\de V)$,
%$[v(1_2\times S^3)]=[\xi P]\in H_3(V)$.
%(c) $z^2\cap W_-\underset{\mod d}\equiv2i_{W_-}(A_f[N]\cap V)\in H_4(W_-,\de)$.
%(c) follows by (a,b).

(b) $p_W^*=2m[t]\in H_4(W,\de)$ for some $m\in\Z$.
\end{Lemma}

Lemma \ref{l:res}.b is essentially proved in the proof of \cite[Lemma 4.8]{CS16}.
%\ref{l:prdiffbor}.

\begin{proof}[Proof of (a)]
The equality follows because
$$Z\cap V= (Z\cap W_-)\cap V= (A_f[N]\times I)\cap V= A_f[N]\cap V\in H_4(V,\de)\quad\text{and}$$
$$\de(A_f[N]\cap V)=A_f[N]\cap \de V=
i_{\de V}(A_f[N]\cap\im v)-i_{\de V}(A_f[N]\cap\nu_f^{-1}P)\overset{(3)}= [\de\Delta]-[\xi P].$$
Here $P$ and $v$ are defined in \cite[Proof of Lemma 4.8]{CS16}.
%before Lemma \ref{l:prdiffbor}.
Equality (3) follows because

$\bullet$ $A_f[N]\cap\nu_f^{-1}P=[\xi P]$ by \cite[Lemma 3.2.a]{CS16}.
%the Section Lemma \ref{l:sec};

%S^2_f used!

$\bullet$ $A_f[N]\cap\im v=[v(1_2\times S^3)]=[\de\Delta]$ since
$$(A_f[N]\cap\im v)\capM{\im v}[v(S^2\times1_3)] = A_f[N]\capM{C_f} v(S^2\times1_3)\overset{(2)} = A_f[N]\capM{C_f}S^2_f=1.$$
Here equality (2) holds because $v(S^2\times1_3)$ is homologous to $S^2_f$ in $C_f$.
\end{proof}

\begin{Lemma}\label{l:prdiff2} For every $y\in K_{u,l}$
%(omitted: and $W,W_-,z,t$ defined in \cite[Proof of Lemma 4.8]{CS16})
%before Lemma \ref{l:prdiffbor}
there is a class $\widehat{z^2}\in H_4(W;\Z_d)$ such that

%(0) $\widehat{z^2}\cap W_-=\rho_dz^2\cap W_-\in H_4(W_-,\de;\Z_d).$
(a) $\overline{z^2}:=\widehat{z^2}+n[t]\in j_W^{-1}\rho_dz^2\subset H_4(W,\de;\Z_d)$ for some $n\in\Z_d$;

(b) $[t]^2=(\widehat{z^2})^2=0\in\Z_d$ and $[t]\capM{W}\widehat{z^2}=2\in\Z_d$.
\end{Lemma}

%{\bf (Such $\widehat{z^2}$ cannot come from $H_4(W';\Z_d)$ because $2[\xi y]\ne0\in H_3(\de W';\Z_d)$!)}
%Note that  $\widehat{z^2}$ is defined by $\widehat{z^2}\cap W_-=\rho_dz^2\cap W_-\in H_4(W_-,\de;\Z_d)$
%up to addition of $[t]$,
%and is uniquely defined by adding the properties $\widehat{z^2}\capM{W}?^*\widehat{z^2}=0$.

The proof is given later in this section.\aronly{\footnote{Equality (3) from \cite[Proof of Lemma 4.3.a]{CS16}
%\ref{l:welleta'}
also holds by \cite[Proof of Lemma 4.7]{CS16} and
%\ref{l:prdiffbor}
by Lemma \ref{l:prdiff2}.a for $d=2$ because
$\eta'(\id\de C_{f_0},Y_{f_0,y})= \overline{z^2}\capM{W} z^2=\overline{z^2}\capM{W}\overline{z^2} =2[t]\capM{W}\widehat{z^2}=0\in\Z_2$.}}
%by Lemma \ref{l:diff'}. $\eta'(\id\de C_f,Y_{f,y})=0$ when $d$ is even

\begin{proof}[Proof of Lemma \ref{l:diff}.a]
The lemma follows by \cite[Lemma 4.8]{CS16} and Lemmas
%\ref{l:prdiffbor},
\ref{l:res}.b, \ref{l:prdiff2}.
Indeed,
$$\overline{z^2}\capM{W}(z^2-p_W^*)=
\overline{z^2}\capM{W}\overline{z^2}-\overline{z^2}\capM{W} p_W^*\overset{(2)}= (\widehat{z^2}+n[t])^2-(\widehat{z^2}+n[t])\capM{W}2m[t]\overset{(3)}=4n-4m.$$
Here

$\bullet$ equality (2) holds by Lemma \ref{l:res}.b and property (a) of Lemma \ref{l:prdiff2},

$\bullet$ equality (3) holds by property (b) of Lemma \ref{l:prdiff2}.
\end{proof}

%Thus in the proof of Lemma \ref{l:diff}.a we use only the statement of Lemma \ref{l:prdiff2}.
In the proof of Lemma \ref{l:diff}.c we will use not only the statement of Lemma \ref{l:prdiff2} but also
the following definition, which is also used in the proof of Lemma \ref{l:prdiff2}.

\smallskip
{\bf Definition of $a,s,\widehat{z^2}$ for $y\in K_{u,l}$.}
By Lemma \ref{l:res}.a there is a representative
$$a\in C_4(V)\quad\text{of}\quad Z\cap V\in H_4(V,\de)\quad\text{such that}\quad\de a=\de\Delta-\xi P.$$
(Such a representative is obtained from a representative $a'\in C_4(V)$ of $Z\cap V\in H_4(V,\de)$ such that $\de a'=\de\Delta-\xi P+\de a''$ for some $a''\in C_4(\de V)$ by the formula $a:=a'-a''$.)

Since $y\in K_{u,l}$, by \cite[Lemma 3.2.$\lambda$]{CS16}
%\ref{l:La}.$\lambda$
there is a chain
$$s\in C_4(C_f\times0;\Z_d)\quad\text{such that}\quad\de s=2\xi P\times0.$$
Define
$$\widehat{z^2}:=[2a-2\Delta-2\xi P\times[0,\frac12]+s]\in H_4(W;\Z_d).$$

\begin{proof}[Proof of Lemma \ref{l:prdiff2}]
We have
$$\rho_dz^2\cap W_-\overset{(1)}=2\rho_di_{V,W_-}(Z\cap V)=
[2a]_{W_-,\de}=[2a-2\xi P\times[0,\frac12]+s]_{W_-,\de}=\widehat{z^2}\cap W_-=j_W\widehat{z^2}\cap W_-,$$
where equality (1) follows by Lemma \ref{l:resa}.
Hence by the cohomology exact sequence of
the pair $(W,W_-)$ (cf.~diagram (*) in \cite[Proof of Lemma 4.8]{CS16})
%before Lemma \ref{l:prdiffbor})
$\rho_dz^2=j_W(\widehat{z^2}+n[t])$ for some $n\in\Z_d$.
Thus property (a) holds.

Let us prove property (b).
We have $[t]^2=[S^2\times 0\times S^2]\capM{S^2\times D^3\times S^2}[S^2\times 1_3\times S^2]=0$.
%Since  $v_3(S^2\times 1_4\times S^2)\subset\partial W_-$, we have $t\cap H_4(W_-)=0$. Hence $t^2=[t]^2_{W_-}???=0$.
Since the support of $\widehat{z^2}$ is in $W'\cup\de C_f\times[0,\frac12]\cup C_f\times0$ and this space is the boundary of a connected component of $W-W'$, we have $(\widehat{z^2})^2=0$.
Also
$$[t]\capM{W}\widehat{z^2}=[t]\capM{W_-} (\widehat{z^2}\cap W_-)=[t]\capM{W_-}[2a]_{W_-,\de}=
%\overset{(2)}=[t]\cap_{W_-} (\rho_dz^2\cap W_-)=[t]\cap_{\de W_-}\de(\rho_dz^2\cap W_-)\overset{(4)}=
2[t]\capM{\de W_-}[\de a]=2[t]\capM{t\times\de\Delta}[\de \Delta]=2.$$
Here the homology classes are taken in the space indicated under `$\cap$'
(so $[t]$ has different meaning in different parts of the formula),
and $\widehat{z^2}\cap W_-=[2a]_{W_-,\de}$ is proved in the proof of (a).
\end{proof}

\begin{proof}[Proof of Lemma \ref{l:diff}.c]
Take any bundle isomorphism $\varphi \colon \de C_0\to\de C_1$ given by \cite[Lemma 2.5]{CS16}.
%Lemma \ref{l:piso}.
Take a closed oriented 3-submanifold $P\subset N$ realizing $y\in H_{f_0}=H_{f_1}$.
For $k=0,1$ construct the maps $v_{jk}$, $j=0,1,2,3$, manifolds $V_k\subset C_k$, $\widehat{V_k}$, $W_k'$ and $W_k$,
chains $a_k,s_k$ and classes $Z_k,z_k,\widehat{z_k^2}$ as in \cite[Proof of Lemma 4.8]{CS16}
%before Lemma \ref{l:prdiffbor}
and above.
(So unlike in other parts of this paper, subscript 0 of a manifold does not mean deletion
of a codimension 0 ball from the manifold.)
%Denote $V:=V_0\cup_{\varphi:\nu_0^{-1}P\to \nu_1^{-1}P} V_1$.
%This is a 6-submanifold of $M_\varphi$ and $\de V=v_{0,0}(S^2\times S^3)\sqcup v_{1,0}(S^2\times S^3)$.
Define
$$W := W_0\cup_{\varphi\times \id I:\de C_0\times I\to\de C_1\times CS16}W_1.$$
The manifold $W$ just defined should not be confused with
the manifolds which were previously denoted $W$ but are now denoted $W_0$ and $W_1$.
The same remark holds for $z,Z,\widehat V$ constructed below.

Consider the following segment of the (`cohomological') Mayer-Vietoris sequence:
$$H_6(W,\de)
\xrightarrow{~r_{W_0} \oplus r_{W_1}~}
H_6(W_0,\de)\oplus H_6(W_1,\de)
\xrightarrow{~r_0\oplus(-r_1)~}
H_4(\de C_0).$$
Here $r_k$ is the composition
$H_6(W_k,\de)\xrightarrow{~\de~} H_5(\de W_k)\xrightarrow{~r_{\de C_0}~} H_4(\de C_0)$.
We have
$$r_kZ_k= (\de Z_k)\cap\de C_0= Y_{f_k}\cap\de C_0=\de(Y_{f_k}\cap C_k)\overset{(4)}=
\de A_k[N]\overset{(5)}=\de A_{1-k}[N]\overset{(6)}= r_{1-k}Z_{1-k}\in H_4(\de C_0).$$
Here

$\bullet$ equality (4) holds by descriptions of of joint Seifert classes \cite[Lemma 3.13.a]{CS16};
%\ref{l:desei}.a.

$\bullet$ equality (5) holds by agreement of Seifert classes \cite[Lemma 3.5.a]{CS16};
%\ref{l:Agr}.a).

$\bullet$ equality (6) holds analogously to the previous set of equalities.

Hence there exists $Z\in H_6(W,\de)$ such that $Z\cap W_k=Z_k$.
Denote
$$\widehat V:=\widehat V_0\bigcup\limits_{\varphi\colon \nu_0^{-1}P\to\nu_1^{-1}P}\widehat V_1\subset W'\quad\text{and}\quad
z:=Z+[\widehat V]\in H_6(W,\de).$$
Clearly,  $z\cap W_k=z_k$.\aronly{\footnote{Note that it is not assumed either that $(W,z)$ is a spin null-bordism of anything or that $\rho_d\de z^2=0$.}}

Take $\widehat{z_k^2}\in H_4(W;\Z_d)$ given by Lemma \ref{l:prdiff2}.
Then by Lemmas \ref{l:res}.b and \ref{l:prdiff2}
$$\overline{z_k^2}\capM{W} p_W^*=4m_k=\widehat{z_k^2}\capM{W} p_W^*\quad\text{and}\quad
\overline{z_k^2}\capM{W} z_k^2=\overline{z_k^2}\capM{W}\overline{z_k^2}=4n_k=
2\widehat{z_k^2}\capM{W} \overline{z_k^2}=2\widehat{z_k^2}\capM{W} z_k^2.$$
Hence
$$\eta(f_k,y)=\rho_{\widehat d}(\widehat{z_k^2}\capM{W_k}(2z_k^2-p_{W_k}^*))=
\rho_{\widehat d}(\widehat{z_k^2}\capM{W}(2z^2-p_W^*)).$$
Take a weakly unlinked section $\xi_0\colon N_0\to\de C_0$ of $f_0$.
By \cite[Lemma 3.4]{CS16}
%\ref{l:phi-unl}
$\xi_1:=\varphi\xi_0$ is an unlinked section of $f_1$.
Hence
$$\de a_1-\de\Delta_1=-\xi_1P=-\xi_0P=\de a_0-\de\Delta_0\quad\text{and}\quad \de s_1=2\xi_1P=2\xi_0P=\de s_0.$$
Identify $M_\varphi$ with $M_\varphi\times0\subset\de W$ and subsets of $M_\varphi$
with the corresponding subsets of $W$.
Denote
$$\widehat a:=[\Delta_0-a_0+a_1-\Delta_1]\in H_4(\widehat V;\Z_d)
%,\quad W'_-:=W'_{0,-}\cup_\varphi W'_{1,-}
\quad\text{and}\quad s:=[s_0-s_1]\in H_4(M_\varphi;\Z_d).$$
%Define $W':= W_0'\cup_\varphi W_1'$. Define $i':W'\to W$ as $i_0'\cup i_1'$.
Then by the definition of $\widehat{z_k^2}$
$$\widehat{z_0^2}-\widehat{z_1^2}=i_\varphi s-2i\widehat a,
\quad\text{where}\quad i_\varphi:=i_{M_\varphi,W}\quad\text{and}\quad  i:=i_{\widehat V,W}.$$
We have $i_\varphi s\capM{W} p_W^* =s\capM{M_\varphi} p_{M_\varphi}^*=0$.

Since
$$(z\cap M_\varphi)\capM{M_\varphi}S^2_{f_0}=(\de z_0\cap C_0)\capM{C_0}S^2_{f_0}=Y_{f_0}\capM{C_0}S^2_{f_0}=1,$$
$z\cap M_\varphi$ is a joint Seifert class for $\varphi$.
Then
$$i_\varphi s\capM{W} z^2=(s\cap\de C_0)\capM{\de C_0}(z^2\cap\de C_0)\overset{(2)}=
2\xi_0 y\capM{\de C_0}\nu_0^! \beta= 2\beta\capM{N} y,$$
where

$\bullet$ $\beta\in H_1(N;\Z_d)$ is a lifting of $\beta(f_0,f_1)$;

$\bullet$ equality (2) follows because we have $s\cap\de C_0=2[\xi_0 P]=2\xi_0 y$ and because we have the identity
$z^2\cap\de C_0=(z\cap M_\varphi)^2\cap\de C_0=\nu_0^! \beta$ by the definition of $\beta(f_0,f_1)$.

We have
$$z^2\capM{W} i\widehat a \overset{(1)}=
(Z+[\widehat V])^2\capM{W} i\widehat a\overset{(2)}=
Z^2\capM{W} i\widehat a+2Z\capM{W}[\widehat V]\capM{W} i\widehat a\overset{(3)}=$$
$$=(Z\cap\widehat V)^2_{\widehat V}\capM{\widehat V}\widehat a+2i(Z\cap\widehat V)\capM{W} i\widehat a\overset{(4)}=
(\widehat a)^3_{\widehat V}+2(i\widehat a)^2\overset{(5)}=(\widehat a)^3_{\widehat V},$$
where

$\bullet$ equality (1) follows by the definition of $z$;

$\bullet$ equalities (2) and (5) follow because $\widehat V\subset W'$, so $[\widehat V]^2=0$ and $(i\widehat a)^2=0$;

$\bullet$ equality (3) is obvious;

$\bullet$ equality (4) follows because $Z\cap\widehat V=\widehat a$ by the definition of $a_0,a_1,\widehat a$.

Therefore
$i\widehat a\capM{W}(2z^2-p_W^*) = 2(\widehat a)^3_{\widehat V}-\widehat a\cap_{\widehat V}p_{\widehat V}^* \underset{12}\equiv0$ by \cite[Theorem 5]{Wa66}.

Now the lemma follows because
$$(\widehat{z_0^2}-\widehat{z_1^2})\capM{W}(2z^2-p_W^*)=
2i_\varphi s\capM{W} z^2-i_\varphi s\capM{W} p_W^*-2i\widehat a\capM{W}(2z^2-p_W^*)
\underset{24}\equiv 4\beta\capM{N} y. $$
\end{proof}

\comment

%Let us prove equality (2); other equalities are obvious.
%We may assume that $W'_-\subset M_\varphi$.
%By `cohomology' exact sequence of pair intersection with $W'_-$ induces an epimorphism
%$H_5(M_\varphi)\to H_5(W'_-,\de)$.
%Hence there is a class $Y\in H_5(M_\varphi)$ such that $Y\cap W'_-=z\cap W'_-$.
%Then  $Y\cap S^2_f=z\cap(S^2_{f_0}\times\frac12)=1$.
%prove
%So $Y$ is a joint Seifert class for $\varphi$.
%Therefore $\rho_d(z\cap W'_-)^2=\rho_d(Y\cap W'_-)^2=\rho_d(Y^2\cap W'_-)=i_{W'_-}\nu_0^!\beta$
%by definition of $\beta(f_1,f_0)$.

%We need??? $[s_k]\cap_W^* [t_k]=0$; this holds because by general position $[2\xi_k P]=0\in C_f-v(S^2\times*)$.

to eta-invariant

We have
$$H_s(\de C_1\times I,U_1)\overset\ex\cong H_s(\Sigma_1,\de)\cong H^{7-s}(\Sigma_1)=0\quad\text{for}\quad s=2,3,4,6.$$
Hence from the exact sequence of pair $(\de C_1\times I,U_1)$ we obtain that the inclusion
$H_s(U_1)\to H_s(\de C_1\times I)$ is an isomorphism for $s=3,4$, a monomorphism for $s=5$ and
an epimorphism for $s=6$.
Looking at the inclusion-induced mapping of the exact sequences of pairs
$(M_\varphi,C_0)$ and $(V,C_0\times I)$, by the 5-lemma we obtain that the inclusion $M_\varphi\to V$
induces an isomorphism in $H_s(\cdot)$ for $s=3,4,5,6$ (for $s=5$ latter using 5-lemma in an improved form).
How to arrive to $r_{M_\varphi}\de$?

{\it Attempt of the proof of the second assertion.}
Since $H_4(S^2\tilde\times S^5)=0$, analogously to \cite[Proof of the Independence Lemma]{Sk08'} and
Lemma \ref{l:cobordbeta} we obtain that for every $x\in H_3$

{\it the triple $(M_{\varphi'},Y_4',Y_3')$ is cobordant to $(M_\varphi,Y_4,Y_3)\sqcup(S^2\tilde\times S^5, A,0)$ for some joint Seifert classs $Y_s\in H_{s+1}(M_\varphi)$ and $Y_s'\in H_{s+1}(M_{\varphi'})$, where $Y_4=Y$ and $Y_3$, $Y_3'$ are for $x$.}

Since $Y_4=Y$ is a $d$-class, we have $\rho_d((Y_4')^2\cap Y_3')=\rho_d(Y_4^2\cap Y_3+A^2\cap0)=0$
for every $x\in H_3$ and some $Y_3'\in H_4(M_\varphi,C_1)$ for $x$.
Hence??? $\rho_dj_{M_\varphi,C_0}(Y_4')^2=0\in H_3(M_\varphi,C_0)$.
Therefore $(Y_4')^2\in \im i_{M_\varphi,C_0}$.
Although $\beta(f_0,f_1)=0$, analogously to the well-definition of $\beta$ (Lemma \ref{l:webeta})
we do not obtain that $\rho_d(Y_4')^2=0$.
Now, we obtain that for every $x\in H_3$ there is a ? $Y_3'$ for $x$ such that $\rho_d((Y_4')^2\cap Y_3')=0$. But this does not imply that $\rho_d(Y_4')^2=0$!
Unlike the proof of the transitivity!

 $r_{M_\varphi}\de:H_s(V,\de)\to H_{s-1}(M_\varphi)$ is an isomorphism for $s=4,6$.
 can also be proved by mapping Mayer-Vietoris sequences induced by $r_{M_\varphi}\de$.)

Consider (the Poincar\'e dual of) the Mayer-Vietoris sequence for $V$:
$$H_6(U_1,\de)\overset{i_{V,U_1}}\to H_6(V,\de)\overset{r_{V,C_0\times I}\oplus r_{V,C_1\times I}}
\to H_6(C_0\times I,\de)\oplus H_6(C_1\times I,\de)
\overset{-r_{\de C_1\times I,U_1}\overline\alpha\de\oplus r_{\de C_1\times I,U_1}\de}\to H_5(U_1,\de).$$
 %,\quad\psi_1(x):=(\de x)\cap U\quad\text{and}\quad \psi_0(x):=(\de\overline\alpha x)\cap U.$$
By the agreement of Seifert classs (Lemma \ref{l:Agr}.a) $\varphi\partial A_0=\partial A_1$ and
$\varphi'\partial A_0=\partial A_1$.
Hence there is a class
$$\overline Y\in H_6(V,\partial)\quad\text{such that}\quad
\overline Y\cap (C_k\times I)=A_k[N]\times I\in H_6(C_k\times I,\partial)\quad\text{for each}\quad k=0,1.$$
Then by  \ref{l:desei}.a $Y:=\de\overline Y\cap M_\varphi$ and $Y':=\de\overline Y\cap M_{\varphi'}$ are joint Seifert classs.
The choice of $\overline Y$ is in $H_6(U_1,\de)\cong H^1(U_1)$.
We have
$$H^s(\de C_1\times I,U_1)\overset\ex\cong H^s(\Sigma_1,\de)\cong H_{7-s}(\Sigma_1)=0\quad\text{for}\quad s=1,2.$$
Hence from the exact sequence of pair $(\de C_1\times I,U_1)$ we obtain that the restriction
$H^1(\de C_1\times I)\to H^1(U_1)$ is an isomorphism.
In dual form this means that $r_{\de C_1\times I,U_1}:H_6(\de C_1\times I,\de)\to H_6(U_1,\de)$ is an isomorphism.

For every $y\in H_3$ let $\overline Y_y:=\overline Y+i_{V,U_1}(\nu^!y\times I)\cap U_1$.
Analogously to Lemma \ref{l:desei}.c
$\overline Y_y^2-\overline Y^2=2i_{V,U_1}(\nu^!y\overline{\lambda(f_0)}(y)\times I)\cap U_1$...

\endcomment

\aronly{

\section{Idea of the proof of Lemma \ref{l:diff}}\label{s:idea}

Here we present Lemma \ref{l:etaiy} which we include for expositional purposes.
For $N = S^1 \times S^3$, this lemma introduces the constructions used in the proof of Lemma \ref{l:diff}.
It can also be used to simplify the proof of Theorem \ref{t:s1s3sm}.
However, we do not present this simplification here.
Hence Lemma \ref{l:etaiy} is not used in the remainder of the paper.

The standard embedding $\tau_0\colon S^1\times S^3\to S^7$ is defined in \cite[\S2.4]{CS16}.

\begin{Lemma}\label{l:etaiy} $\eta(\tau_0,y)=0$ for every $y\in H_3(S^1\times S^3)$.
\end{Lemma}

\begin{proof}
Define an extension
$$\inc\colon D^2\times D^4\to S^7\quad\text{of $\tau_0$ by}\quad \inc(x,y):=(y\sqrt{2-|x|^2},0,0,x)/\sqrt2.$$
%Then $\inc|_{S^1\times S^3}=\tau_0$.
Take an embedding $v_0\colon S^5\to S^7-\inc(S^1\times D^4)$ whose linking coefficient with
$\inc(S^1\times D^0)$ is equal to   $y\capM{S^1\times S^3}[S^1\times1_3]$.
We omit the subscript $\tau_0$ in this proof.
As
$\lk(\widehat Ay,\tau_0(S^1\times1_3))=y\capM{S^1\times S^3}[S^1\times1_3]$, we have
$\widehat Ay=[v_0(S^5)]\in H_5(C)\cong\Z$.
We also have $\widehat Ay=i_{C}\nu^!y$.
Take a representative $P$ of $y$ and a chain
$$
V\in C_6(C)\quad\text{such that}\quad \de V=\nu^{-1}P-v_0(S^5).
$$
Since $C$ is parallelizable, $v_0$ extends to an embedding $v_2\colon S^5\times D^2\to\Int C=\Int C\times\frac12$
which is orientation-preserving and transversal to $V$ and such that $\im v_2\cap V=v_0(S^5)$.
Extend $v_2$ to an orientation-preserving embedding $v_3\colon S^5\times D^3\to \Int(C\times I)$.
Let
$$
W_-:= C\times I-\Int\im v_3\quad\text{and}\quad W:=W_-\cup_{v_3|_{S^5\times S^2}}D^6\times S^2.$$
Consider the cohomology exact sequence of the pair $(W,W_-)$ in the following Poincar\'e dual form
(analogous to the sequence (*) in \cite[Proof of Lemma 4.8]{CS16}):
%\S\ref{s:welleta'}):
$$\xymatrix{
H_6(D^6\times S^2
%t\times\Delta
) \ar[r]  & H_6(W,\de) \ar[r]^{r_{W_-}} & H_6(W_-,\de) \ar[r] & H_5(
%t\times\Delta
D^6\times S^2) \\
H^2(W,W_-)\ar[u]_{\cong}^{PD\circ\ex}  & & & H^3(W,W_-)\ar[u]_{\cong}^{PD\circ\ex}
}.$$
Since $H_5(D^6\times S^2)=0$, the map $r_{W_-}$ is an epimorphism.
Take any
$$Z\in r_{W_-}^{-1}(A[N]\times I\cap W_-)\subset H_6(W,\partial).$$
Denote
$$\widehat V:=V\cup D^6\times1_2\quad\text{and}\quad z:=Z+[\widehat V]\in H_6(W,\partial).$$
%z:=Z+[V+D^6\times1_2]\in H_6(W,\partial).$$
Since $H_5(D^6\times S^2)=0$, the spin structure on $W_-$ coming from $S^7\times I$ extends to $W$.
Clearly, $\partial W\underset{spin}=\partial(C\times I)\underset{spin}=M$
(for the `boundary' spin structure on $M$ coming from $C\times I$).
Since
$$\de_WZ=\de_{C\times I}(A[N]\times I)=Y_0\quad\text{and}\quad
\de_W[\widehat V]=[\nu^{-1}P\times\frac12]=i_M\widehat Ay,
\quad\text{we have}\quad \de_W z=Y_y.$$
By \cite[Lemma 4.7]{CS16}
%Lemma \ref{l:simpleY0}
$\de_Wz^2=Y_y^2=0$.
So $z^2\in\im j_W$.
Analogously to (*) we obtain an isomorphism $H_4(C\times I)\cong H_4(W)$ commuting with
$i_{C\times I}\colon H_4(M)\to H_4(C\times I)$ and $i_W\colon H_4(M)\to H_4(W)$.
Since $i_{C\times I}$ is onto, $i_W$ is onto.
Hence $j_W=0$.
Thus $z^2=0$.
So we take $\overline{z^2}:=0$ and obtain that $\eta(\tau_0,y)=\overline{z^2}\capM{W}(z^2-p_W^*)=0$.\footnote{We essentially proved that if $\widehat A_fy$ is spherical, then $\eta(f,y)=0$.}
\end{proof}

\section{Discussion of the action of knots}\label{s:action}

\begin{Remark}[The action of knots in Theorem \ref{t:gensm}] \label{r:at}
(a) Take any $[f]\in E^7(N)$.
Let
$$
O(f)=O([f]):=\{[f\#g]\ :\ [g]\in E^7(S^4)\}
$$
be  the orbit of $[f]$ under the action of $E^7(S^4)$.
We have $O(f)=\beta_{u,l}^{-1}(b)$ when $[f]\in \beta_{u,l}^{-1}(b)$ by \cite[Theorem 1.2]{CS16} and
the additivity of $\varkappa,\lambda$ and $\beta$ \cite[Lemmas 2.3 and 2.9]{CS16}.

Define the {\it inertia group} of $f$, $I(f) \subset E^7(S^4)=\Z_{12}$,\footnote{The inertia group of $f$ is just the stabilizer of $[f]$ under the action of $E^7(S^4)$.
We use the word `inertia' following its use for the action of the group homotopy spheres on
the diffeomorphism classes of smooth manifolds: see the last paragraph of Remark \ref{r:general}.}
to be the subgroup of isotopy classes in $E^7(S^4)$ which do not change
$[f]$ after embedded connected sum:
$$
I(f)=I([f])  := \{ [g]\in E^7(S^4) \, : \, [f \#g]=[f]\}
$$
%We have $I(f) \cong \im\theta_{u, l, b}$ when $[f]\in \beta_{u,l}^{-1}(b)$ by Theorem \ref{t:gensm},
%additivity of $\lambda,\varkappa,\beta$ \cite[Lemmas 2.3 and 2.9]{CS16} and Addendum \ref{l:addeta}.
%As described in \S\ref{s:intro-int}, the problem at hand is to determine $I(f)$.
For some cases this orbit and group are found in terms of $u,l,b$ in Corollaries \ref{t:actioni} and \ref{c:gen}.

(b) The indeterminancy in the classification of Theorem \ref{t:gensm}.c corresponds to the fact that we do not always know  $\im\theta_{u,l,b}$.
Thus determining $\im\theta_{u, l, b}$ becomes a key problem.
This image is found in this paper when either $u=0$ or $2\rho_d\overline l=0$ (Corollary \ref{c:gen}) or in the cases (1,2,3) of Remark \ref{r:cordif}.
%(mind that $\beta_{u,l}$ depends on the choice of $f'$, see end of \S\ref{s:defandplan-b}).

For general $u,l$ and simple enough $N$ there are some $u,l$ such that for each $b$ the methods of this paper do not completely determine $\im\theta_{u,l,b}$.
E.g.  let $N = (S^1 \times S^3) \# (S^2 \times S^2)$.
Then $x \capM{N}y\capM{N} u = 0$ for each $x, y \in H_3$ and $u \in H_2$.
Take the standard bases for $H_2 \cong H_2(S^2 \times S^2)$ and for $H_3 \cong H_3(S^1 \times S^3)$.
The pair $((0,6), l)$ is symmetric, where $l(x,y):=xy$.
So by Theorem \ref{t:gensm}.a there is an embedding $f\colon N\to S^7$ such that $(\varkappa\times\lambda)(f)=((0,6),l)$.
Then $d:=\di (0,6)=6$, $\widehat d=6$ and $\overline l:H_3\to H_1$ is `the identity'.
Hence $\rho_d\overline l$ is surjective.
Thus
$C_{u,l}=\frac{H_1\otimes\Z_d}{2H_1\otimes\Z_d}\cong \Z_2$,
$K_{u,l}=3H_3\cong\Z$ and the pairing $\capM{d}\colon C_{u,l}\times K_{u,l}\to\Z_d$ is
given by $\rho_2z\capM{d}3y=3\rho_d(z\capM{N} y)$ for each $z\in H_1$ and $y\in H_3$.
So this pairing is trivial mod3.
Then by Theorem \ref{t:gensm}.c $\theta_{u, l, b} = \theta_{u, l, b'}$ for each $b, b' \in C_{u,l}$.
Both the trivial and the non-trivial homomorphisms $K_{u,l}\to 4\Z_6\cong\Z_3$ fit into the conclusion of
Theorem \ref{t:gensm}.c, so Theorem \ref{t:gensm}.c does not allow us to completely determine $\im\theta_{u,l,b}$.
%DELETE? Note that in the above case $f = f_0 \sharp f_1$ for embeddings
%$f_0 \colon S^1 \times S^3 \to S^7$ and $f_1 \colon S^2 \times S^2 \to S^7$.
%It seems natural to try prove and additivity formula for $\theta_{u, l, b}$ in this case,
%leading us to conjecture that in this case we can find $b \in C_{u,l}$
%such that $\theta_{u, l, b} = 0$.  We leave this for future research.}
\end{Remark}

\begin{Problem}
(a) Characterize those closed connected orientable 4-manifold with torsion free $H_1$ such that for every embeddings $f:N\to S^7$ and $g\colon S^4\to S^7$ the embeddings $f\#g$ and $f$ are isotopic (in other words, the action $\#$ of $E^7(S^4)$ on $E^7(N)$ is trivial).
%, cf. \S\ref{s:action}).

Cf. Corollary \ref{t:fa}.
If $H_1=0$, then this property is equivalent to $H_2^{DIFF}$ containing no elements divisible either by 4 or by  3.
E.g. $N=\C P^2$ satisfies this property (because $\sigma(\C P^2)=1$).

(b) Characterize those $f$ for which $|O(f)|=12$ (i.e.~$|I(f)|=1$),
and those $f$ for which $|O(f)|=1$ (i.e.~$|I(f)|=12$).
\end{Problem}

%{\bf Remark on smooth structures.}
%Let $N_{TOP}$ denote the topological manifold underlying $N$ and let $N_\alpha$ denote a smoothing of $N_{TOP}$.
%It is very natural to ask how the set $E^7(N_\alpha)$ depends upon the choice of the smooth structure $\alpha$.
%We leave this question for future work, but propose the following

%Let $(W; N_0, N_1)$ be an $h$-cobordism.  Then
%there is a geometrically defined bijection $E(W)\colon E^7(N_0) \to E^7(N_1)$ depending only
%on the diffeomorphism type of $W$ relative to its boundary.

\begin{Addendum} \label{c:s1s3}
For every $l\in\Z-\{0\}$ there is a map
%homomorphism???
$\psi_l\colon \Z\times\Z_{2l}\to\Z_{12}$ such that for every
$a,a'\in\Z_{12}$ and $l,b,l',b'\in\Z$ we have  $a\#\tau(l,b)=a'\#\tau(l',b')$ if and only if
$$
\left[\begin{matrix}\text{either}\quad l=l'=0, \quad b=b'\quad\text{and}\quad a\equiv a'\mo 2\gcd(b,6),\\
\text{or}\quad l=l'\ne0, \quad b\equiv b'\mo2l \quad\text{and}\quad
\rho_{12}(a-a')=\psi_l([b/2l],\rho_{2l}b)-\psi_l([b'/2l],\rho_{2l}b). \end{matrix}\right.
$$
%a-a'=\dfrac{b-b'}{2l}\psi_l(\rho_{2l}b)\end{matrix}\right..$$
%For this we need $\eta(\tau(l,b)+_p\tau(0,2l),\tau(l,b))=\eta(\tau(l,b')+_p\tau(0,2l),\tau(l,b'))$,
%but is it correct? Cf.~Lemma \ref{l:etau}.
\end{Addendum}

\begin{proof}
%[Proof of Addendum \ref{c:s1s3}]
By Theorem \ref{t:s1s3sm}.b if either $l\ne l'$ or $b\not\equiv b'\mo2l$, then the equivalence is
clear because neither of the two assertions holds.

Assume that $l=l'$ and $b\equiv b'\mo2l$.
%Let $\tau:=a\#\tau(l,b)$ and $\tau'=a'\#\tau(l,b')$.
Let $\tau:=\tau(l,b)$ and $\tau'=\tau(l,b')$.
By the Isotopy Classification Theorem \ref{l:Is} and Theorem \ref{t:s1s3sm}.b
%and calculation of $\beta$ \cite[Lemma 2.7.b]{CS16}
$a\#\tau=a'\#\tau' ~\Leftrightarrow~ \eta(a\#\tau,a'\#\tau')=0$.
Since $\di\varkappa(\tau)=0$, we may use Corollary \ref{c:gen}.b.

If $l=0$, then $b=b'$.
By Theorem \ref{t:s1s3sm}.b and Corollary \ref{c:gen}.b $\im\theta_{0,0,b}$ consists of elements of $\Z_{24}$
which are divisible by $4\gcd(b,6)$.
Hence by Addendum \ref{l:addeta} and the transitivity of $\eta$ (Lemma \ref{l:difeta})
$$\eta(a\#\tau,a'\#\tau')=\rho_{4\gcd(b,6)}(2a-2a')\in\Z_{4\gcd(b,6)}.$$
If $l\ne0$, then by Theorem \ref{t:s1s3sm}.b and Corollary \ref{c:gen}.b $\im\theta_{0,l,b}=0$ and
$\eta(a\#\tau,a'\#\tau')\in 2\Z_{24}$.
Identify $\Z_{2l}$ and $\{0,1,\ldots,2l-1\}$.
For every $x\in\Z_{2l}$ and $k\in\Z$ define
$$\psi_l(k,x):=\frac12\eta(\tau(l,x),\tau(l,x+2kl))\in\Z_{12}.$$
Define
%$\overline b\in\Z_{2l}$ by $b\equiv \overline b\mo2l$ and
$\overline\tau:=\tau(l,\rho_{2l}b)$.
Then by Addendum \ref{l:addeta} and the transitivity of $\eta$ (Lemma \ref{l:difeta})
$$
\frac{\eta(a\#\tau,a'\#\tau')}2=
\rho_{12}(a-a') - \frac{\eta(\overline\tau,\tau)}2 + \frac{\eta(\overline\tau,\tau')}2=
\rho_{12}(a-a') - \psi_l([b/2l],\rho_{2l}b) + \psi_l([{b'}/{2l}],\rho_{2l}b).
$$
%\begin{multline*}
%\eta(\tau,\tau')=
%\rho_{24}(2a-2a') - \eta(\tau(l,\overline b),\tau(l,b)) + \eta(\tau(l,\overline b),\tau(l,b'))= \\
%\rho_{24}(2a-2a') - \psi_l([\frac b{2l}],\rho_{2l}b) + \psi_l([\frac{b'}{2l}],\rho_{2l}b),
%\end{multline*}
The two formulas for $\eta(a\#\tau,a'\#\tau')$ above imply the stated equivalence.
\end{proof}

\begin{Remark}[An approach to the action of knots]\label{r:action}
Let us explain the ideas required to move from the classification modulo knots in \cite{CS16} to the
main results of this paper.
We briefly recall and continue the discussion in \cite[1.4]{CS16}.

Suppose that $f_0, f_1 \colon N \to S^7$ are embeddings.
Assume that $f_1$ is isotopic to $f_0 \# g$ for some embedding $g \colon S^4 \to S^7$.
By \cite[Isotopy Classification Modulo Knots Theorem 2.8]{CS16} this is equivalent to $\lambda(f_0)=\lambda(f_1)$, $\varkappa(f_0)=\varkappa (f_1)$ and $\beta(f_0,f_1)=0$.
The complements $C_0$ and $C_1$ may be glued together along a bundle isomorphism
$\varphi\colon\de C_0\to\de C_1$ to form a parallelizable closed $7$-manifold $M = C_0 \cup_{\varphi} (-C_1)$.
% with $H_*(M)$ torsion free. Alexander duality and the Mayer-Vietoris theorem imply that
%there is an isomorphism $H_5(M) \cong \Z \oplus H_3$.
%Choose a {\it joint Seifert class} $Y\in H_5(M)$.
%\begin{equation} \label{eq:Y}
%Y = 1 \oplus x \in H_5(M),
%\end{equation} for some $x \in H_3$.
Recall that $d: = \di\varkappa(f_0)$ is the divisibility of $\varkappa(f_0) \in H_2$.
By the assumption on $f_0,f_1$ there is a joint Seifert class $Y\in H_5(M)$ such that $\rho_dY^2=0$, i.e.~{\it a $d$-class} \cite[Lemma 4.1]{CS16}.
There is a spin null-bordism $(W, z)$ of $(M_\varphi, Y)$, since $\Omega_7^{Spin}(\C P^\infty) = 0$.
%and we can determine the coset $[g] \in E^7(S^4)/I(f_0)$ by computing a certain characteristic
%number of $(W, z)$.
Since $\rho_dY^2 = 0$, the class $\rho_dz^2 \in H_4(W, \de;\Z_d)$ lifts to $\overline{z^2} \in H_4(W;\Z_d)$.
Recall that $p^*_W \in H_4(W, \de)$ is the Poincar\'{e} dual of $p_W$, the spin Pontrjagin class of $W$.
We then verified that {\it the Kreck invariant},
%\begin{equation} \label{eq:eta}
$$
\eta(\varphi,Y) := \overline{z^2}\capM{W}\rho_{\widehat d}(z^2-p^*_W) \in \Z_{\widehat d},
$$
%\end{equation}
determines the surgery obstruction for $W$ to be spin diffeomorphic to the product $C_0\times I$ \cite{CS16}.
%equivalence class of $g \in E^7(S^4)/I(f)$.
%This lead to the classification of $E^7(N)$ when $H_1 = 0$.
We proved that $\eta(\varphi,Y)$ is independent of the choices of $W,z,\overline{z^2}$ for
a fixed bundle isomorphism $\varphi$ and $d$-class $Y$ \cite[\S4.1]{CS16}.
We also proved that $\eta(\varphi,Y)$ is independent of the choice of $\varphi$: for the precise statement, see
\cite[Lemma 4.3.c]{CS16}.
So we need to know the various values of $\eta(\varphi,Y)$ arising from the different possible choices of $Y$.
These choices are described in \cite[Description of $d$-classes Lemma 4.7]{CS16}.
The achievement of this paper is showing that the change of $\eta(\varphi,Y)$ under a change of $Y$ is precisely determined by $\theta_{u, l, b}$, and proving the properties of $\theta_{u, l, b}$ (Lemma \ref{l:diff}).
\end{Remark}

It is well known that there are pairs of $4$-manifolds $N_0$ and $N_1$ which are homeomorphic
but not diffeomorphic (see e.g.~\cite{GS99}).
It is natural to ask about the relationship between $E^7(N_0)$ and $E^7(N_1)$ in this case.

\begin{Conjecture} Let $N_0$ and $N_1$ be closed connected
%orientable
$4$-manifolds.
%with $H_1(N_0; \Z)$ and $H_1(N_1; \Z)$ torsion free.

(a)
A homeomorphism $h \colon N_0 \to N_1$ gives rise to a
%is a homeomorphism then there is a
geometrically defined bijection
\linebreak $E_h \colon E^7(N_0) \to E^7(N_1)$.
%which depends only on the topological isotopy class of $h$.

(b)
A smooth $h$-cobordism $(W; N_0, N_1)$ gives rise to a geometrically defined
bijection \linebreak
$E_{(W; N_0, N_1)} \colon E^7(N_0) \to E^7(N_1)$.
%depending only on the diffeomorphism type of $W$ relative to its boundary.
\end{Conjecture}

}

%\newpage

%\aronly{
{\it In this list books, surveys and expository papers are marked by stars}

\end{document}